\documentclass [10pt]{amsart}

\usepackage{amsmath}
\usepackage{amsfonts}
\usepackage{amstext}
\usepackage{amsbsy}
\usepackage{amsopn}
\usepackage{amsxtra}
\usepackage{upref}
\usepackage{amsthm}
\usepackage{amsmath}
\usepackage{amssymb}
\usepackage{tikz}
\usetikzlibrary{automata,positioning}
\usepackage{enumitem}
\usepackage{epsfig,verbatim,ifthen,graphicx}
\usepackage[all]{xy}
\usepackage{epsfig,verbatim,ifthen,graphicx}

\newtheorem{prop}{Proposition}[section]
\newtheorem{lemma}{Lemma}[section]
\newtheorem{thm}{Theorem}[section]
\newtheorem{ex}{Example}[section]
\newtheorem{coro}{Corollary}[section]

\theoremstyle{definition}
\newtheorem{defi}{Definition}[section]
\newtheorem{rem}{Remark}[section]
\usepackage{enumitem}

\def\R{{\mathbb R}}

\def\N{{\mathbb N}}

\def\F{{\mathcal F}}

\def\G{{\mathcal G}}

\begin{document}
\title
{Relative pressure functions and their equilibrium states} 

\author{Yuki Yayama}
\address{Centro de Ciencias Exactas and Grupo de investigaci\'{o}n en Sistemas Din\'{a}micos y Aplicaciones-GISDA, Departamento de Ciencias B\'{a}sicas,  Universidad del B\'{i}o-B\'{i}o, Avenida Andr\'{e}s Bello 720, Casilla 447, Chill\'{a}n, Chile}
\email{yyayama@ubiobio.cl}

\begin{abstract}
For a subshift $(X, \sigma_X)$ and a subadditive sequence 
$\F=\{\log f_n\}_{n=1}^{\infty}$ on $X$,  we study equivalent conditions for the existence of $h\in C(X)$ such that $\lim_{n\rightarrow\infty}(1/{n})\int \log f_n d \mu=\int h d \mu$ for every invariant measure $\mu$ on $X$. 
For this purpose, we first we study necessary and sufficient conditions 
for  $\F$ to be an asymptotically additive sequence in terms of certain properties for periodic points. 
For a factor map $\pi: X\rightarrow Y$, where $(X, \sigma_X)$ is an irreducible shift of finite type and  $(Y, \sigma_Y)$ is a subshift,
applying our results and the results obtained by Cuneo \cite{Cu} on asymptotically additive sequences,  we study the existence of $h$ with regard to a subadditive sequence associated to a relative pressure function. 
This leads to a characterization of the existence of a certain type of continuous compensation function for a factor map between subshifts. 
 As an application, we study 
 the projection $\pi\mu$ of an invariant weak Gibbs measure $\mu$ for a continuous function on an irreducible shift of finite type. 
 \end{abstract}

\maketitle
  
\section{Introduction}\label{intro}
The thermodynamic formalism for sequences of continuous functions generalizes 
the formalism  for continuous functions   and it has been applied to solve some dimension problems in nonconformal dynamical systems. 
The equilibrium states for sequences of continuous functions are the equilibrium states  for Borel measurable functions in general.
In \cite{Fa} Falconer introduced the thermodynamic formalism for subadditive sequences to study repellers of nonconformal transformations.
Cao, Feng and Huang in \cite{CFH} established the theory for subadditive sequences wherein the variational principle was obtained for compact dynamical systems.
Asymptotically additive sequences,  which generalize almost additive sequences studied by Barreira \cite{b2} and Mummert \cite {m}, were also introduced by Feng and Huang \cite{FH}.
The properties of equilibrium states for sequences of continuous functions, such as uniqueness, the (generalized) Gibbs property and  mixing properties have been also  studied (see for example, \cite{b2, m, Fe1}).  
Here, a natural question arises. 

\smallskip
\paragraph*{\underline{Question 1}} Given a subadditive sequence 
$\F=\{\log f_n\}_{n=1}^{\infty}$ on a compact metric space $X$, what  are necessary and sufficient conditions for the existence of a continuous function $h$ on $X$  such that 
\begin{equation}\label{question}
\lim_{n\rightarrow\infty}\frac{1}{n}\int \log f_n d \mu=\int h d \mu
\end{equation}
 for every invariant measure $\mu$ on $X$? 
 
 \smallskip
If such an $h$ exists, then the thermodynamic formalism for such sequences $\F$ reduces to the formalism for continuous functions. Cuneo \cite[Theorem 1.2]{Cu} 
 proved that if a sequence of continuous functions is 
 asymptotically additive (see (\ref{defasi}) for the definition), then there always exists $h\in C(X)$ satisfying (\ref{question}) for every invariant measure $\mu$ on $X$.  
In this paper, we study necessary conditions for a subadditive sequence $\F$ on an irreducible subshift $(X, \sigma_X)$
 to have a continuous function $h \in C(X)$ satisfying (\ref{question})  for every invariant measure $\mu$ on $X$. Using our results and the result obtained by Cuneo \cite[Theorem 1.2]{Cu}, we give some answers to Question 1 (Theorems \ref{mainaa},   \ref{kohiyo} and \ref{weakgibbs}).
 Towards this end, we first study conditions for a subadditive sequence on a subshift to be an asymptotically additive sequence in terms of certain properties for periodic points. 
Given a subadditive sequence $\F=\{\log f_n\}_{n=1}^{\infty}$ on $X$, 
if (\ref{question}) holds for every invariant Borel probability measure $\mu$ on $X$, then the sequence
 $\tilde\F=\{(1/n)\log (f_n/e^{S_nh})\}_{n=1}^{\infty}$ converges (pointwise) to the zero function $0$ for every periodic point of $\sigma_X$ (see  Proposition \ref{main11}). We show in Theorems \ref{mainaa} and \ref{mainas} that if the sequence
 $\tilde\F$ converges (pointwise) to $0$ for every periodic point of $\sigma_X$
and $\F$ satisfies a particular property for certain periodic points then $\tilde \F$ converges to $0$ everywhere, moreover, it converges uniformly to $0$ on $X$. This gives the asymptotic additivity of $\F$. 
We apply Theorem \ref{mainaa}
when we study Question 1 with regard to a relative pressure function of a continuous function (Theorems \ref{kohiyo} and \ref{weakgibbs}).  In Proposition \ref{main11}, Question 1 is studied in a general form.
 Note that subadditive sequences are not asymptotically additive in general (see Example \ref{exshin} in Section  \ref{apli}).

 
In Section \ref{relation}, we consider relative pressure functions in relation to compensation functions. 
 Let $(X, \sigma_X), (Y, \sigma_Y)$ be subshifts and  
$\pi:X\rightarrow Y$ be a factor map.
Let $f\in C(X)$, $n\in \N$ and $\delta>0$. 
For each $y\in Y$, define
$$P_n(\sigma_X, \pi, f, \delta)(y)=\sup\{\sum_{x\in E}e^{(S_nf)(x)}:E \text { is an } (n, \delta) \text{ separated subset of } \pi^{-1}(\{y\})\},$$
$$P(\sigma_X, \pi, f, \delta)(y)=\limsup_{n\rightarrow \infty}\frac{1}{n}\log P_{n}(\sigma_X, \pi, f,\delta)(y),$$
$$P(\sigma_X, \pi, f)(y)=\lim_{\delta\rightarrow 0}P(\sigma_X, \pi,f, \delta)(y).$$
The function $P(\sigma_X, \pi, f):Y\rightarrow \R$ is  the {\em relative pressure} function of $f\in C(X)$ with respect to $(\sigma_X, \sigma_Y, \pi)$. In general it is merely Borel measurable. 
In Theorem \ref{kohiyo}, for an irreducible shift of finite type $(X, \sigma_X)$, we study equivalent conditions for a relative pressure function $P(\sigma_X, \pi, f)$ on $Y$ 
to have a function $h\in C(Y)$ such that 
\begin{equation}\label{mainq}
\int P(\sigma_X, \pi, f)d \mu=\int h d\mu \textnormal{ for every } \mu\in M(Y, \sigma_Y)\end{equation}
 where $M(Y, \sigma_Y)$ is the set of invariant Borel probability measures on $Y$. In general, a relative pressure function $P(\sigma_X, \pi, f)$ is represented by a subadditive sequence $\G=\{\log g_n\}_{n=1}^{\infty}$ of continuous functions on $Y$ (see (\ref{beq}) for $g_n$), that is, $P(\sigma_X, \pi, f) =\lim_{n\rightarrow\infty}(1/n)\log g_n$ 
almost everywhere with respect to every $\mu \in M(Y, \sigma_Y)$. The sequence $\G$ satisfies an additional condition (see \ref{a4} in Section \ref{seqmany}) weaker than almost additivity and it is not asymptotically additive in general.  
We prove that the subadditive sequence $\G$ on $Y$ associated to $P(\sigma_X, \pi, f)$ satisfies the particular property for certain periodic points in Lemma \ref{mainaa} \ref{ms2}.  Applying Theorem \ref{mainaa}  we obtain in Theorem \ref{kohiyo} that,
for $h\in C(Y)$, uniform convergence of $\tilde \G=\{(1/n)\log (g_n/e^{S_nh})\}_{n=1}^{\infty}$ to 0 on $Y$ is equivalent to pointwise convergence of $\tilde \G$ to $0$ for every periodic point of $\sigma_Y$. In particular, we obtain that (\ref{mainq}) holds if and only if the sequence $\G$ associated to $P(\sigma_X, \pi, f)$ is asymptotically additive. 
Moreover, if there exists an invariant weak Gibbs measure $m$ for $f\in C(X)$, then 
(\ref{mainq}) holds if and only if $\pi m$ is an invariant weak Gibbs measure for some continuous function on $Y$ (Theorem \ref{weakgibbs}). The properties of the sequence $\G$ associated to $P(\sigma_X, \pi, f)$ under the existence of $h$ in (\ref{mainq}) are studied and a condition of nonexistence of such a continuous function is also studied (Corollary \ref{chara1}). 
These results are applied directly to study the 
projection of an invariant weak Gibbs measure for a continuous function on $X$ in Section \ref{apli} (see Theorem \ref{general} and Corollary \ref{cc1}). Note that in general if there exists an invariant weak Gibbs measure $m$ for $f\in C(X)$, then $\pi m$ is a weak Gibbs equilibrium state for the subadditive sequence $\G$ associated to $P(\sigma_X, \pi, f)$.

On the other hand, relative pressure functions are connected with compensation functions. Given $f\in C(X)$, Theorem \ref{kohiyo} relates the question on the existence of $h$ in (\ref{mainq}) with the existence of a compensation function $f-h\circ\pi$ for some $h\in C(Y)$.
A function $F\in C(X)$ is a \em{compensation function} for a factor map $\pi$ if 
\begin{equation}
\sup_{\mu\in M(X, \sigma_X)}\{h_{\mu}(\sigma_X)+\int F d\mu+\int \phi \circ \pi d\mu\}=\sup_{\nu\in M(Y, \sigma_Y)}\{h_{\nu}(\sigma_Y)+\int \phi d\nu\}
\end{equation}
for every $\phi\in C(Y)$. 
If $F=G\circ \pi$, $G\in C(Y)$, then $G\circ \pi$ is a \em{saturated compensation function}. 
The concept of compensation functions was introduced by Boyle and Tuncel \cite{BT} and their properties  were studied by Walters \cite{Wa} in relation to relative pressure. 
The existence of compensation functions 
has been studied \cite{A, S1, Sh2, Sh3}. Shin \cite{Sh2, Sh3} proved a saturated compensation function does not always exist  and gave a characterization for the existence of a saturated compensation function for factor maps between shifts of finite type. 
A function $-h\circ \pi\in C(X)$ is a saturated compensation function if and only if (\ref{mainq}) holds for $f=0$. Our results connect the result obtained by Shin 
with the asymptotic additivity of the sequence associated to $P(\sigma_X, \pi, 0)$(see Remark \ref{comments} and Corollary \ref{A1}).
Since saturated compensation functions were applied to study the measures of full Hausdorff dimension of nonconformal repellers,  studying the properties of equilibrium states for $h$ in (\ref{mainq}) would help in the further study of certain dimension problems (see Example \ref{yukithesis}). 

Section \ref{prep} deals with a particular class of  subadditive sequences on subshifts 
 satisfying an additional property  (see condition \ref{a1}) weaker than almost additivity but stronger than \ref{a4}. 
 The result of Feng \cite[Theorem 5.5]{Fe1} implies
that there is a unique (generalized) Gibbs equilibrium state for a subadditive sequence with bounded variation satisfying  the property \ref{a1}. 
We study equivalent conditions for this type of sequence $\F=\{\log f_n\}_{n=1}^{\infty}$ on a subshift $X$ to have a continuous function for which the unique Gibbs equilibrium state is a weak Gibbs measure (Theorem \ref{wk=g}). In this case, 
we obtain that for $h\in C(X)$ uniform convergence of  the sequence of functions $\{(1/n)\log (f_n/e^{S_nh})\}_{n=1}^{\infty}$ to 0 on $X$ is equivalent to pointwise convergence  of  the sequence  to 0 on $X$. 
 We note that it is not clear that  the condition for certain periodic points in Theorem \ref{mainas} \ref{mss2} is satisfied for this type of sequence in general. 


\section{Background}
\subsection{Shift spaces}
We give a brief summary of the basic definitions in symbolic dynamics. 
$(X, \sigma_X)$ is a {\em one-sided subshift} if $X$ is a closed
shift-invariant subset of $\{1,\dots, k\}^{\N}$ for some $k\geq
1$, i.e., $\sigma_X(X)\subseteq X$,  where the shift 
$\sigma_X:X\rightarrow X$
is defined by $(\sigma_X(x))_{i}=x_{i+1}$ for all $i\in \N$, $x=(x_n)^{\infty}_{n=1} \in X.$
Define a metric $d$ on $X$ by $d(x,x')={1}/{2^{k}}$ if
$x_i=x'_i$ for all $1\leq i\leq k$ and $x_{k+1}\neq {x'}_{k+1}$, $d(x,x')=1$ if $x_1\neq x'_1$,
and $d(x,x')=0$ otherwise. Throughout this paper, we consider one-sided subshifts.
Define a cylinder set $[x_1 \dots x_{n}]$ of length $n$ in $X$  by 
$[x_1\dots x_n]=\{(z_i)_{i=1}^{\infty} \in X:  z_i=x_i \text{ for all }1\leq i\leq n\}.$  For each $n \in \N,$ denote by
$B_n(X)$ the set of all $n$-blocks that appear in points in $X$.
Define $B_{0}(X)=\{\epsilon\},$ where  $\epsilon$ is the empty word of length $0$.
The language of $X$ is the set $B(X)=\cup_{n=0}^{\infty}B_n(X)$. A subshift $(X,\sigma_X)$ is {\it irreducible} if for any allowable words $u, v\in B(X)$, there exists  $w\in B(X)$ such that $uwv \in B(X)$.  A subshift has the {\it weak specification property} if there exists $p\in\N$ such that for any allowable words $u, v \in B(X)$, there exist  $0\leq k\leq p$ and $w\in B_{k}(X)$ such that $uwv\in B(X)$. We call such $p$ a weak specification number. 
A point $x\in X$ is a periodic point  of $\sigma_X$ if there exists $l\in\N$ such that $\sigma_X^{l}(x)=x$.

Let $(X, \sigma_X)$ and $(Y, \sigma_Y)$ be subshifts. 
A shift of finite type $(X,\sigma_X)$ is {\em one-step} if there exists a set $F$ of forbidden blocks of length $\leq 2$ such 
that $X=\{x\in \{1, \dots, k\}^{\N}: \omega \textnormal{ does not appear in } x \textnormal{ for any }\omega\in F\}$. 
A map
$\pi:X\rightarrow Y$ is a {\it factor map} if it is continuous, surjective and satisfies $\pi \circ \sigma_{X} = \sigma_Y\circ \pi$. 
If in addition, the $i$-th position of the image of
$x$ under $\pi$ depends only on $x_i,$ then $\pi$ is a {\em one-block
factor map}. 
Throughout the paper we assume that a shift of finite type $(X, \sigma_X)$ is one-step and $\pi$ is a one-block factor map.  Denote by $M(X, \sigma_X)$ the collection of all $\sigma_X$-invariant Borel probability measures on $X$
and by ${Erg}(X, \sigma_X)$ all ergodic members of $M(X, \sigma_X)$. 

\subsection{Sequences of continuous functions.}\label{seqmany}
We give a brief summary on the basic results on the sequences of continuous functions considered in this paper.  
Let $(X, \sigma_X)$ be a subshift on finitely many symbols.  For each $n\in\N$, let $f_n: X\rightarrow \R^{+}$ be a continuous function.
A sequence $\F=\{\log f_n\}_{n=1}^{\infty}$ is {\em almost additive} if there exists a constant $C\geq 0$ such that 
$e^{-C}f_n(x) f_{m}(\sigma^n_X x)  \leq  f_{n+m}(x) \leq e^{C}f_n(x) f_{m}(\sigma^n_X x) $.
In particular, if $C=0$, then $\F$ is additive. 
The thermodynamic formalism for almost additive sequences was studied in Barrera \cite{b2} and Mummert \cite{m}. More generally, Feng and Huang \cite{FH} introduced  asymptotically additive sequences which generalize almost additive sequences.
A sequence $\F=\{\log f_n\}_{n=1}^{\infty}$ is {\em asymptotically additive} on $X$ if for every $\epsilon > 0$ there exists a continuous function $\rho_{\epsilon}$ such that
\begin{equation}\label{defasi}
\limsup_{n\rightarrow \infty} \frac{1}{n} 
\lVert  \log f_{n}-S_n\rho_{\epsilon} \rVert_{\infty} < \epsilon,
\end{equation}
where $\lVert \cdot \rVert_{\infty}$ is the supremum norm and $(S_n\rho_{\epsilon})(x)=\sum_{i=0}^{n-1}\rho_{\epsilon}(\sigma^i(x))$ for each $x\in X$.
A sequence $\mathcal{F}= \{ \log f_n \}_{n=1}^{\infty}$ is {\it subadditive} if $\F$ satisfies $f_{n+m}(x) \leq  f_n(x) f_{m}(\sigma^n_X x)$. The thermodynamic formalism for subadditive sequences was studied by Cao, Feng and Huang \cite{CFH}. 


We assume certain regularity conditions on sequences.
A sequence $\mathcal{F}= \{ \log f_n \}_{n=1}^{\infty}$ has {\it bounded variation}  if there exists $M \in \R^{+}$ such that
 $\sup \{ M_n : n \in \N\} \leq M$ where
 \begin{equation}\label{forfn}
 M_n= \sup \left\{ \frac{f_n(x)}{f_n(y)} : x,y  \in X, x_i=y_i \textrm{ for } 1 \leq i \leq n\right\}.
\end{equation}
More generally, if $\lim_{n \rightarrow \infty}(1/n)\log M_n=0$, then we say that $\mathcal{F}$ has  {\it tempered variation}. Without loss of generality, we assume $M_{n}\leq M_{n+1}$ for all $n\in \N$. 

A function $f\in C(X)$  belongs to  the {\it Bowen class} if the sequence $\F$ by setting $f_n= e^{S_n(f)}$
has bounded variation \cite{w3}. A function of summable variation belongs to the Bowen class. In this paper,  we consider 
the sequences $\F$ satisfying the following properties.
\begin{enumerate}[label=(C\arabic*)]
\item \label{a0} The sequence $\F':=\{\log (f_ne^C)\}_{n=1}^{\infty}$ is subadditive for some $C\geq 0$. 
\item There exist $p\in\N$ and $D>0$ such that  given any $u \in B_n(X), v \in B_m(X)$, $n, m\in\N$, there exist $0\leq k\leq p$ and $w \in B_k(X)$ such that \label {a1} 
\begin{equation*}
\sup \{f_{n+m+k}(x):x \in [uwv]\} \geq D \sup \{f_n(x):x\in [u]\} 
\sup \{f_m(x):x \in [v]\}.
\end{equation*}
\end{enumerate}
More generally, 
\begin{enumerate}[label=(D\arabic*)]
\setcounter{enumi}{1}
\item   There exist $p\in\N$ and a positive sequence $\{D_{n,m}\}_{(n,m)\in \N \times \N}$ such that  given any $u \in B_n(X), v \in B_m(X)$, $n, m\in\N$, there exist $0\leq k\leq p$ and $w \in B_k(X)$ such that  \label{a4}
$$\sup \{f_{n+m+k}(x):x \in [uwv]\} \geq D_{n,m} \sup \{f_n(x):x\in [u]\} 
\sup \{f_m(x):x \in [v]\},$$
where $\lim_{n\rightarrow\infty}(1/n)\log D_{n,m}=\lim_{m\rightarrow\infty}(1/m)\log D_{n,m}=0$.  Without loss of generality, we assume that $D_{n,m}\geq D_{n,m+1}$ and $D_{n,m}\geq D_{n+1,m}$.
\end{enumerate}
\begin{rem}
 A sequence $\F=\{\log f_n\}_{n=1}^{\infty}$ satisfying \ref{a0} is not always asymptotically additive (see Section \ref{apli}). 
 The condition \ref{a1} was introduced by Feng \cite{Fe} where 
 the thermodynamic formalism of products of matrices was studied. 
The sequences satisfying \ref{a0} and \ref{a1} with bounded variation generalize almost additive sequences
with bounded variation  on subshifts with the weak specification property and they  
has been applied to solve questions concerning the Hausdorff dimensions of nonconformal repellers \cite{Fe1, Y2}. See \cite{KR, ily} for the noncompact case. We will study the sequences satisfying  \ref{a0} and   \ref{a4} in Sections \ref{relation} and \ref{apli}. 
\end{rem}
Let $(X,\sigma_X)$ be a subshift and  $\F=\{\log f_n\}_{n=1}^{\infty}$ be a subadditive sequence of continuous functions on $X$. 
For each $n\in\N$, define $Z_n(\F)$ by $Z_n(\F)=\sum_{i_1\cdots i_n \in B_n(X)} \sup \{f_{n}(x): x\in [i_1\cdots i_n]\}.$ Then the {\em  topological pressure } for $\F$ is defined by 
\begin{equation}\label{defp}
P(\F) =\limsup_{n\rightarrow\infty} \frac{1}{n}\log Z_n(\F).
\end{equation}

\begin{thm}\cite{CFH} \label{VPS}
Let $(X,\sigma_X)$ be a subshift and  $\F=\{\log f_n\}_{n=1}^{\infty}$ be a subadditive sequence on $X$. 
Then 
\begin{equation}\label{vs}
P(\F) =\sup_{\mu\in M(X,\sigma_X)}\{h_{\mu}(X)+\lim_{n\rightarrow\infty}\frac{1}{n}\int\log f_nd\mu\}.
\end{equation}
\end{thm}

A measure $m\in M(X, \sigma_X)$ is an {\em equilibrium state} for $\F$  if 
the supremum  in (\ref{vs}) is attained at $m$. 

\begin{defi}
Let $(X,\sigma_X)$ be a subshift and $\F=\{\log f_n\}_{n=1}^{\infty}$ be a subadditive sequence on $X$ satisfying $P(\F)\neq -\infty$.
A measure $\mu\in M(X, \sigma_X)$ is a {\em weak Gibbs measure} for $\F$ if there exists $C_n>0$ such that 
\begin{equation*}
\frac{1}{C_n}< \frac{\mu[x_1\dots x_n]}{e^{-nP(\F)}f_n(x)}<C_n
\end{equation*}
where $\lim_{n\rightarrow\infty}(1/n)\log C_n=0$,
for every $x\in X$ and $n\in \N$.
If there exists $C>0$ such that $C=C_n$ for all $n\in \N$, then $\mu$ is a {\em Gibbs measure}.
\end{defi}



If $\mu$ is an invariant weak Gibbs measure for a subadditive sequence $\F$, then it is an equilibrium state for $\F$. The result of Feng \cite[Theorem 5.5]{Fe1} implies the uniqueness of equilibrium states for a class of sequences satisfying \ref{a0} and \ref{a1}.   
\begin{thm}\cite{Fe1} \label{fengthem}
Let $(X,\sigma_X)$ be a subshift and $\F=\{\log f_n\}_{n=1}^{\infty}$  be a sequence on $X$ satisfying 
\ref{a0} and \ref{a1} with bounded variation. Then there is a unique invariant Gibbs measure for $\F$ and it is the unique equilibrium state for $\F$.
\end{thm}


Cuneo \cite{Cu} showed that finding equilibrium states for asymptotically additive sequences is equivalent to that for continuous functions. 

\begin{thm}\label{cuneo}Special case of \cite[Theorem 1.2]{Cu}.
Let $(X,\sigma_X)$ be a subshift and $\F=\{\log f_n\}_{n=1}^{\infty}$ be an asymptotically additive sequence on $X$. Then there exists $f\in C(X)$ such that 
\begin{equation}\label{cuneothm}
\lim_{n\rightarrow \infty} \frac{1}{n} 
\lVert \log f_{n}-S_nf \rVert_{\infty}=0.
\end{equation}
 \end {thm}
Hence if $\F$ is asymptotically additive, then there exists $f\in C(X)$ such that 
$\lim_{n\rightarrow \infty} (1/n) \int \log f_{n} d\mu=\int f d\mu$ 
 for every $\mu\in M(X,\sigma_X).$ It is clear that (\ref{cuneothm}) implies that $\F$ is asymptotically additive.





. 

\section{subadditive sequences}
In this section, we consider Question 1 from Section \ref{intro}.  Proposition \ref{main11}  is valid for the case when $X$ is a compact metric space and $T: X\rightarrow X$ is a continuous transformation of $X$.  Proposition \ref{main11} will be applied in the next sections. 
\begin{prop}\label{main11}
Let $(X,\sigma_X)$ be a subshift and $\F=\{\log f_n\}_{n=1}^{\infty}$  a subadditive sequence on $X$. 
For $h\in C(X)$, the following conditions are equivalent.

\begin{enumerate}[label=(\roman*)]
\item  $$ \lim_{n\rightarrow\infty}\frac{1}{n}\int \log f_n d\mu=\int hd\mu$$
for every $\mu\in M(X, \sigma_X).$\label{mm1}

\item $$ \lim_{n\rightarrow\infty}\frac{1}{n}\int \log f_n d\mu=\int hd\mu$$
for every $\mu\in Erg(X, \sigma_X).$\label{mm2}

\item $$ \lim_{n\rightarrow\infty}\frac{1}{n}\log \left(\frac{f_n(x)}{e^{(S_nh)(x)}}\right)=0$$
$\mu$-almost everywhere on $X$, for every  $\mu\in Erg(X, \sigma_X)$.\label{mm3}

\end{enumerate}

\end{prop}

\begin{rem} Proposition \ref {main11} holds for a sequence $\F$ satisfying \ref{a0} because $\{\log (e^Cf_n)\}_{n=1}^{\infty}$ is a subadditive sequence. 
 \end{rem}
\begin{proof}
It is clear that \ref{mm1} implies \ref{mm2}. By the ergodic decomposition (see \cite[Proposition A.1 (c)]{FH}), \ref{mm2} implies \ref{mm1}. Now we assume that \ref{mm2} holds. For a measure $\mu\in Erg(X,\sigma_X)$, we obtain 
\begin{equation*}
 \lim_{n\rightarrow\infty}\frac{1}{n}\int \log \left(\frac{f_n}{e^{(S_nh)}}\right) d\mu=0.
 \end{equation*}
 To see that this implies  \ref{mm3},
define $r_n(x):=f_n(x)/e^{(S_nh)(x)}$. Then $\log r_1\in L^{+}_{1}(\mu)$ and 
$\{\log r_n\}_{n=1}^{\infty}$ is a subadditive sequence of continuous functions on $X$. Since $\mu$ is an ergodic measure, by Kingman's subadditive ergodic theorem,  we obtain that 
$\lim_{n\rightarrow\infty}(1/n)\log r_n(x)=\inf_{n\in \N}(1/n)\int \log r_n d\mu=0$ $\mu$-almost everywhere on $X$. 
Now we assume that \ref{mm3} holds. 
Given $\mu\in Erg(X,\sigma_X)$, applying the subadditive ergodic theorem to the sequence $\{\log r_n\}_{n=1}^{\infty}$, we obtain  
 \begin{equation*}
 \begin{split}
 \int \lim_{n\rightarrow\infty}\frac{1}{n}\log \left(\frac{f_n}{e^{(S_nh)}} \right) d\mu&=\lim_{n\rightarrow\infty} \frac{1}{n}\int \log \left( \frac{f_n}{e^{(S_nh)}} \right)d\mu
 \\ &=\lim_{n\rightarrow\infty} \left(\frac{1}{n}\int \log {f_n}d\mu-
\int h d\mu\right).
 \end{split}
 \end{equation*}
 Hence we obtain \ref{mm2}.
\end{proof}

\section{Subadditive sequences which are asymptotically additive}\label{subas}
 Subadditive sequences are not always asymptotically additive. In this section we study a class of subadditive sequences on shift spaces (compact spaces) which are also asymptotically additive. The goal of this section is to characterize such sequences using a particular property for periodic points. The results in this section are applied in 
Sections \ref{relation} and \ref{apli} to study relative pressure functions.
   
\begin{lemma} \label{maina} 
Let $(X,\sigma_X)$ be a subshift and $\F=\{\log f_n\}_{n=1}^{\infty}$ be a sequence on $X$ satisfying \ref{a0} with tempered variation. 
 Suppose that $\F$ satisfies the following two conditions \ref{ms1} and \ref{ms2}.

 \begin{enumerate}[label=(\roman*)]
\item There exists  $h\in C(X)$ such that  $$ \lim_{n\rightarrow\infty}\frac{1}{n}\log 
\left(\frac{f_n(x)}{e^{(S_nh)(x)}}\right)=0$$
for every periodic point $x\in X$. \label{ms1}

\item There exist $k, N\in \N$ and a sequence $\{M_n\}_{n=1}^{\infty}$ of positive real numbers satisfying $ \lim_{n\rightarrow\infty}({1}/{n})\log M_n=0$
such that  for given any $u\in B_n(X), n \geq N$, there exist $0 \leq q \leq k$ and $w\in B_q(X)$ such that $z:=(uw)^{\infty}$ is a point in $X$
 satisfying 
 \begin{equation}\label{eqk1}
f_{j(n+q)}(z)\geq (M_n \sup\{f_n(x): x\in [u]\})^j
\end{equation}
for every $j\in \N$.
\label{ms2}
\end{enumerate}
Then $\F$ is an asymptotically additive sequence on $X$.
\end{lemma}
\begin{rem}
Let $(X,\sigma_X)$ be an irreducible shift of finite type and $k$ be a weak specification number.  Then for each $u\in B_n(X)$ 
there exist $0\leq q\leq k$ and  $w\in B_q(X)$ such that 
$(uw)^{\infty}\in X$.
\end{rem}
\begin{proof}
Suppose that \ref{ms1} and \ref{ms2} hold. We will show that
\begin{equation}\label{kagi}
\lim_{n\rightarrow\infty}\frac{1}{n} \lVert  \log \left(\frac{f_n}{e^{(S_nh)}}\right)\rVert _{\infty}=0.
\end{equation}
Let $k, M_n, N$ be defined as in \ref{ms2}.  For $h\in C(X)$, let
\begin{equation}\label{temph}
M^h_n:=\sup\{\frac{e^{(S_nh)(x)}}{e^{(S_nh)(x')}}: x_i=x'_i, 1\leq i\leq n\}
\end{equation}
for each $n\in \N$ and $C_h:=\max_{0\leq i \leq k}\{(S_ih)(x): x\in X\}$, where $(S_0h)(x):=1$ for every $x\in X$.
Let $\epsilon>0$. Take $N_1\in \N$ large enough so that 
$$\frac{1}{n}\vert \log ({M^h_n}{e^{C_h}})\vert<\epsilon, 
\frac{1}{n}\vert \log M_n\vert<\epsilon \text{ and } \frac{n}{n+k}>\frac{1}{2} $$
 for all $n>N_1$.
Let $N_2=\max\{N, N_1\}$ and let $n\geq N_2$. Then for $x_1\dots x_n\in B_n(X)$, there exists $w\in B_{q}(X), 0\leq q\leq k,$ such that 
$y^{*}:=(x_1,\dots, x_n, w)^{\infty}\in X$ satisfying (\ref{eqk1}). Since $y^{*}$ is a periodic point,  \ref{ms1} implies that there exists $N(y^{*}) \in\N$ such that
$$ \frac{1}{i}\vert \log \left(\frac{f_i(y^{*})}{e^{(S_ih)(y^{*})}}\right)\vert< \epsilon$$
for all $i > N(y^{*})$.
Take
$j> N(y^{*})$. By  \ref{ms2}, for $ z \in [x_1\dots x_n]$, 
we obtain
\begin{equation*}
\begin{split}
\epsilon>&\frac{1}{j(n+q)}
\log \left(\frac{f_{j(n+q)}(y^{*})}{e^{(S_{j(n+q)}h)(y^{*})}}\right)
\geq \frac{1}{j(n+q)}\log 
\left(\frac{M_nf_{n}(z)}{M^h_ne^{(S_{n}h)(z)}e^{C_h}}\right)^j\\
&=\frac{1}{(n+q)}\log M_n
+\frac{1}{(n+q)}\log \left(\frac{f_{n}(z)}{e^{(S_{n}h)(z)}}\right)
-\frac{1}{(n+q)}\log ({M^h_n}{e^{C_h}})\\
&>-2\epsilon +\frac{1}{n+q} \log \left(\frac{f_{n}(z)}{e^{(S_{n}h)(z)}}\right)>-2\epsilon +\frac{n}{n+q} \left (\frac{1}{n} \log \left(\frac{f_{n}(z)}{e^{(S_{n}h)(z)}}\right)\right).
\end{split}
 \end{equation*}
 Without loss of generality assume 
$(1/n)\log(f_n(z)/e^{(S_nh)(z)}) >0$.
Hence for any $x_1\dots x_n\in B_n(X), n\geq N_2, z\in [x_1\dots x_n]$, we obtain that  
$({1}/{n}) \log ({f_{n}(z)}/e^{(S_{n}h)(z)})<6\epsilon.$ 
 
Next we show that there exists $N'\in \N$ such that for all for all $z\in [x_1,\dots, x_n], n\geq N'$,
$({1}/{n}) \log ({f_{n}(z)}/e^{(S_{n}h)(z)})>-4\epsilon.$ 
Since $\F$ has tempered variation, for each $n\in \N$,
 let $M^{\F}_n:=\sup\{f_n(x)/f_n(x'): x_i=x'_i, 1\leq i\leq n\}$.
 Let $C_{\F}:=\max_{0\leq i\leq k}\{f_i(x): x\in X\}$, where $f_0(x):=1$ for every $x\in X$,  and 
 $\bar{C}_h:=\min_{0\leq i\leq k}\{(S_ih)(x): x\in X\}$. Let $C$ be defined as in \ref{a0}.
 Take $N_3\in \N$ large enough so that 
 $$\frac{1}{n}\vert \log ({M^{\F}_n}{M^{h}_n}C_{\F}{e^{-{\bar{C}}_h+2C}})\vert<\epsilon \text{ and } \frac{n}{n+k}>\frac{1}{2}$$ for all $n>N_3$.
Since $\F$ satisfies \ref{a0},  we obtain that 
 \begin{equation*}
\begin{split}
\frac{f_{j(n+q)}(y^{*})}{e^{(S_{j(n+q)}h)(y^{*})}}
\leq & \left(\frac{C_{\F}M^{h}_ne^{2C}\sup\{f_n(y):y\in [x_1\dots x_n]\}}
{e^{\bar{C}_h}\sup\{e^{(S_nh)(y)}: y\in [x_1\dots x_n]\}}\right)^j\\
&\leq  \left(\frac{C_{\F}M^{h}_nM^{\F}_ne^{2C}f_n(z)}
{e^{\bar{C}_h+(S_nh)(z)}}\right)^j,
\end{split}
\end{equation*}
where in the last inequality $z$ is a point from the cylinder set 
$[x_1\dots x_n]$.
Hence for $j> N(y^{*})$,
 \begin{equation*}
\begin{split}
-\epsilon<&\frac{1}{j(n+q)}
\log \left(\frac{f_{j(n+q)}(y^{*})}{e^{(S_{j(n+q)}h)(y^{*})}}\right)\\
&<\frac{1}{n+q} \log \left(\frac{f_{n}(z)}{e^{(S_{n}h)(z)}}\right)
+\frac{1}{n+q} \log ({M^{\F}_n}{M^{h}_n}C_{\F}{e^{-{\bar{C}}_h+2C}})\\
&<\frac{1}{n+q} \log \left(\frac{f_{n}(z)}{e^{(S_{n}h)(z)}}\right)+\epsilon\\
&=\frac{n}{n+q} \left(\frac{1}{n} \log \left(\frac{f_{n}(z)}{e^{(S_{n}h)(z)}}\right)\right)+\epsilon.
\end{split}
\end{equation*} 
Without loss of generality assume $(1/n)\log (f_{n}(z)/e^{(S_{n}h)(z)})<0$. For all $z \in [x_1\dots x_n],n \geq N_3$, 
we obtain that 
$(1/n)\log(f_n(z)/e^{(S_nh)(z)}) >-4\epsilon.$ 
Hence we obtain (\ref{kagi}).
 \end{proof}
 By Lemma \ref{maina}, we obtain some conditions for a sequence $\F$ satisfying \ref{a0} to be asymptotically additive, assuming that Lemma \ref{maina} \ref{ms2} is satisfied. 
 \begin{thm} \label{mainaa}
Let $(X,\sigma_X)$ be a subshift. Let $\F=\{\log f_n\}_{n=1}^{\infty}$ be a sequence on $X$ satisfying \ref{a0}  with tempered variation and Lemma \ref{maina} \ref{ms2}. 
Then  the following statements  are equivalent for $h\in C(X)$.
 \begin{enumerate}[label=(\roman*)]
\item  $\F$ is  asymptotically additive on $X$ satisfying 
$$\lim_{n\rightarrow \infty} \frac{1}{n} \lVert \log 
\Big (\frac{f_n}{e^{(S_n h)}}\Big) \rVert_{\infty}=0.$$
\label{11}
\item 
$$ \lim_{n\rightarrow\infty}\frac{1}{n}\int \log f_n d\mu=\int hd\mu>-\infty$$
for every $\mu\in M(X, \sigma_X).$ \label{13}

\item $$ \lim_{n\rightarrow\infty}\frac{1}{n}\log \left(\frac{f_n(x)}{e^{(S_nh)(x)}}\right)=0$$
for every periodic point $x\in X$. \label{12}

\item $$ \lim_{n\rightarrow\infty}\frac{1}{n}\log \left(\frac{f_n(x)}{e^{(S_nh)(x)}}\right)=0$$
for every  $x\in X$. \label{14}

\end{enumerate}
\end{thm}
\begin{proof}
The implications 
``$\ref{11} \Longrightarrow \ref{13} \Longrightarrow \ref{12}$" are clear by applying Theorem \ref{cuneo} and Proposition \ref{main11}. To see 
``$\ref{12} \Longrightarrow \ref{14} \Longrightarrow \ref{11} $'', we apply 
 Lemma \ref{maina}. 
 \end{proof}

In the next theorem we study an equivalent condition for a subadditive sequence $\F$ to be an asymptotically additive sequence.
 
 \begin{thm} \label{mainas}
Let $(X,\sigma_X)$ be an irreducible subshift of finite type 
and $\F=\{\log f_n\}_{n=1}^{\infty}$ be a sequence on $X$ satisfying \ref{a0} with tempered variation. 
Then $\F$ is asymptotically additive on $X$ if and only if
the following two conditions \ref{mss1} and \ref{mss2} hold.

 \begin{enumerate}[label=(\roman*)]
\item There exists  $h\in C(X)$ such that  $$ \lim_{n\rightarrow\infty}\frac{1}{n}\log (\frac{f_n(x)}{e^{(S_nh)(x)}})=0$$
for every periodic point $x\in X$. \label{mss1}
\item 
\label{mss2}
There exist $k\in \N, c\geq 0$  and a sequence $\{M_n\}_{n=1}^{\infty}$  of positive real numbers  satisfying $\lim_{n\rightarrow\infty}(1/{n})\log M_n=0$ such that the following property (P) holds.\\

(P) For every $0<\epsilon<1$, there exists $N\in \N$ such that for given any $u\in B_n(X), n\geq N$, there exist $0\leq q\leq k$ and  $w\in B_q(X)$ such that
$z:=(uw)^{\infty}$ is a point in $X$ satisfying 
\begin{equation}\label{eqk}
f_{j(n+q)}(z)\geq (M_ne^{-cn\epsilon})^j (\sup\{f_n(x): x\in [u]\})^j
\end{equation}
for every $j\in \N$.

 

\end{enumerate}
Theorem \ref{mainaa} holds if we replace Lemma \ref{maina} \ref{ms2} by the condition \ref {mss2} above.
\end{thm}
\begin{rem}
Theorem \ref{mainas} \ref{mss2} is a generalization of   Lemma \ref{maina} \ref{ms2}. If we set $c=0$ in (\ref{eqk}), we obtain  (\ref{eqk1}).

\end{rem}
\begin{proof}
Assume that $\F$ is asymptotically additive. Then \ref{mss1} is obvious and for a given $0<\epsilon <1$ there exists $N\in\N$ such that for all $n\geq N$
\begin{equation}\label{pas}
e^{-n\epsilon+(S_nh)(x)}<f_n(x)<e^{n\epsilon+(S_nh)(x)}
\end{equation}
for all $x\in X$. 
Since $(X,\sigma_X)$ is an irreducible shift of finite type, let $k$ be a weak specification number. Then  for $x_1\cdots x_n\in B_n(X),  n\geq N$, there exists $w\in B_q(X), 0\leq q\leq k$ such  that $y^*:=(x_1, \dots, x_n, w)^{\infty}\in X$. 
Let  
$ \bar{C}_h, M^{h}_n$ and $M^{\F}_n$ be defined as in the proof of Lemma \ref{maina}. 
Then for any $z\in [x_1\dots x_n], j\in  \N$, 
\begin{equation*}
\begin{split}
f_{(n+q)j}(y^{*})&\geq e^{-j(n+q)\epsilon+(S_{j(n+q)}h)(y^{*})}
\geq e^{-j(n+q)\epsilon}\cdot\left(\frac{1}{M^{h}_n}e^{(S_{n}h)(z)}e^{\bar{C}_h}\right)^j\\
&\geq \left(\frac{1}{M^{h}_n}e^{-2\epsilon n-k\epsilon+\bar{C}_h}\right)^j f_n^j(z)
\geq \left(\frac{1}{M^{h}_n}e^{-2\epsilon n-k+\bar{C}_h}\right)^j f_n^j(z).
\end{split}
\end{equation*}
Setting $c=2$ and $M_n= e^{-k+\bar{C}_h}M^{\F}_n/M^{h}_n$, we obtain \ref{mss2}.
Now we show the reverse implication. 
We slightly modify the proof of Lemma \ref{maina} by taking into account of the property (P).
We only consider the case when $c>0$. 
Let $C_h$ and $M^{h}_n$ be defined as in the proof of Lemma \ref{maina}.
Let $0<\epsilon<1$ be fixed. By \ref{mss2},  there exists $N'\in \N$ such that
$$ -\frac{3c}{2}\epsilon<\frac{1}{n+i}\log (e^{-nc\epsilon}M_n)<-\frac{c}{2}\epsilon, \frac{1}{n}\vert \log ({M^h_n}{e^{C_h}})\vert<\epsilon \textnormal{ and } \frac{n}{n+k}>\frac{1}{2}$$ 
for  all $n>N',  0\leq i\leq k$. 
 In the proof of Lemma \ref{maina}, 
 define $N_2:=\max\{N, N'\}$.
 Replacing $M_n$ by $e^{-nc\epsilon}M_n$ in the proof of Lemma \ref{maina}, we obtain that for any $x_1\dots x_n\in B_n(X), n\geq N_2, z\in [x_1\dots x_n]$
$$ \frac{1}{n} \log \left(\frac{f_{n}(z)}{e^{(S_{n}h)(z)}}\right)<(4+3c)\epsilon.$$
 Using the latter part of the proof of Lemma \ref{maina}, we obtain the results.

\end{proof}




\section{asymptotically additive sequences and subadditive sequences satisfying \ref{a0} and \ref{a1}}\label{prep}
In this section, we study the sequences $\F$ on subshifts $X$ with bounded variation satisfying \ref{a0} and \ref{a1}. 
 Since there exists a unique Gibbs equilibrium state $m$ for  such a sequence $\F$ (Theorem \ref{fengthem}), we study the condition for $m$ to be an invariant Gibbs measure 
 for some continuous function. In Theorem \ref{cc}, we also characterize the form of sequences $\F$ in terms of  the  properties of equilibrium states.   


\begin{thm}\label{wk=g}
Let $(X,\sigma_X)$ be a subshift and $\F=\{\log f_n\}_{n=1}^{\infty}$ be a sequence on $X$ satisfying \ref{a0} and \ref{a1} with bounded variation. Let $m$ be the unique invariant Gibbs measure for $\F$. Then the following statements are equivalent.
\begin{enumerate} [label=(\roman*)]
\item There exists  $h\in C(X)$ such that 
\begin{equation*}
 \lim_{n\rightarrow\infty}\frac{1}{n}\log \left(\frac{f_n(x)}{e^{(S_nh)(x)}}\right)=0
 \end{equation*}
 for every $x\in X$.\label{wk0}
 \item  $\F$ is asymptotically additive on $X$. \label{wka}
  \item The measure $m$ is an invariant weak Gibbs measure for  a continuous function on $X$. 

 \end{enumerate}
\end{thm} 
\begin{rem}\label{imp1}
(1)There exists a sequence $\F$ which satisfies   \ref{a0}, \ref{a1} with bounded variation satisfying Theorem \ref{wk=g} \ref{wka}. On the other hand, there exists a sequence $\F$ with bounded variation satisfying  \ref{a0} and \ref{a1} without being asymptotically additive (see Section \ref{apli}). (2) If $h\in C(X)$ in \ref{wk0} exists, then $m$ is a unique equilibrium state for $h$.

\end{rem}

To prove Theorem \ref{wk=g}, we apply the following lemmas. We continue to use $\F$ and  $m$ defined as in Theorem \ref{wk=g}. In the next lemma we first study the relation between Theorem \ref{wk=g} \ref{wk0} and \ref{wka}. 



\begin{lemma} \label{mainot}
Let $(X,\sigma_X)$ be a subshift and $\F=\{\log f_n\}_{n=1}^{\infty}$ be a sequence on $X$ satisfying \ref{a0}  and  \ref{a1} with bounded variation. 
If there exists  $h\in C(X)$ such that 
\begin{equation}\label{strong}
 \lim_{n\rightarrow\infty}\frac{1}{n}\log \left(\frac{f_n(x)}{e^{(S_nh)(x)}}\right)=0
 \end{equation}
for every $x \in X$, then $\F$ is asymptotically additive on $X$.
 \end{lemma}
\begin{rem}
Lemma \ref{mainot} implies that if $\F$ satisfies the assumptions of the lemma 
 then uniform convergence of the sequence of functions $\{\frac{1}{n}\log \left(\frac{f_n}{e^{(S_nh)}}\right)\}_{n=1}^{\infty}$ is equivalent to pointwise convergence of the sequence of functions. 
\end{rem}
\begin{proof}
Let $\epsilon >0$. It is enough to show that there exists $N\in \N$ such that for any $z \in [u], u\in B_n(X), n>N$, 
\begin{equation*}
-\epsilon<\frac {1}{n}\log \Big (\frac{f_n (z)}{e^{(S_n h) (z)}}\Big)<\epsilon.
\end{equation*}
Let $p$ be defined as in \ref{a1}.
Let $\overline{m}_h:=\max_{0\leq l \leq p}\{e^{(S_{l}h)(x)}: x\in X\}$, where $(S_0h)(x):=1$ for every $x\in X$.  Since $h$ has tempered variation, let $M^{h}_n$ be defined as in (\ref{temph}). 
Let $M$ be a constant defined as in the definition of bounded variation and $D$ be defined as in \ref{a1}. 
Then there exists $N_1\in \N$ such that 
\begin{equation}\label{use1}
\frac{1}{n} \log M <\epsilon, \frac{1}{n} \log M^{h}_n <\epsilon,  \frac{1}{n}\vert \log \frac{1}{\overline{m}_h}\vert<\epsilon,  \frac{1}{n}\vert \log D\vert<\epsilon  \text{ and }
\frac{n}{n+p}>\frac{1}{2}.
\end{equation}
for all $n >N_1$. Take $n>N_1$.
The condition \ref{a1} implies that for a given $u \in B_{n}(X)$,   there exists $w_1\in B_{l_1}(X), 0\leq l_1\leq p$ such that for any $x\in [uw_1u], z\in[u]$
\begin{equation*}
\sup\{f_{2n+l_1}(x): x\in  [uw_1u]\} \geq D
\left(\sup\{f_{n}(x): x\in [u]\}\right)^2\geq D
f^2_{n}(z).
\end{equation*}
Repeating this,  given $j\geq 2, u \in B_{n}(X)$,   there exist allowable words $w_i$ of length $l_i$, $1\leq i\leq j-1, 0\leq l_i\leq p,$
such that $uw_1uw_2u\dots uw_{j-1}u$ is an allowable word of length $jn+\sum_{i=1}^{j-1}l_i$ satisfying that for any $x\in [uw_1uw_2u\dots uw_{j-1}u]$ and $z\in [u]$
\begin{equation}\label{eq11}
Mf_{jn+\sum_{i=1}^{j-1}l_i}(x)\geq \sup\{f_{jn+\sum_{i=1}^{j-1}l_i}(x): x\in [uw_1uw_2u\dots uw_{j-1}u]\}
\geq  D^{j-1}
{f_{n}(z)}^j.
\end{equation}
By the additivity of the sequence $\{e^{S_nh}\}_{n=1}^{\infty}$, 
\begin{equation}\label{eq12}
e^{(S_{jn+\sum_{i=1}^{j-1}l_i}h)(x)}\leq 
\left(M^{h}_ne^{S_{n}h(z)}\right)^j{\overline{m}_h}^{j-1}.
\end{equation}
Hence by (\ref{eq11}) and (\ref{eq12}) we obtain for $j\geq 2$, $x\in [uw_1uw_2u\dots uw_{j-1}u]$ and $z\in [u]$, 
\begin{equation}\label{eq13}
\frac{f_{jn+\sum_{i=1}^{j-1}l_i}(x)}{e^{(S_{jn+\sum_{i=1}^{j-1}l_i}h)(x)}}\geq \left(\frac{1}{M^{h}_n}\right)^{j} \left(\frac{f_{n}(z)}{e^{(S_{n}h)(z)}}\right)^{j} \left(\frac{D}{\overline{m}_h}\right)^{j-1}\cdot \frac{1}{M}. 
\end{equation}

Let $c_1=[u w_1u],\dots, c_i=[uw_1 uw_2u\dots u w_i u], i\in\N$.
Then Cantor's intersection theorem $\cap_{i\in\N} c_i\neq \emptyset$ and it consists of exactly one point in $X$. We call it $x^{*}\in X$. Note that $x^{*}$ may not be a periodic point. 
 For each $y\in X$, define $A_n(y):=f_{n}(y)/e^{(S_{n}h)(y)}$.
 By the assumption (\ref{strong}), there exists $t(x^{*})\in \N$,  which depends on $x^{*}$ such that for all $i\geq t(x^{*})$,
\begin{equation*}
-\epsilon<\frac {1}{i}\log  A_i(x^{*})<\epsilon.
\end{equation*}

Letting $s(u, j):=\sum_{i=1}^{j-1}l_i$, for $j\geq t(x^{*})\geq 2$, and using (\ref{use1}) and (\ref{eq13})
we obtain
\begin{equation*}
\begin{split}
&\epsilon>\frac{1}{jn+s(u, j)} \log A_{jn+s(u, j)} (x^{*}) \\
&\geq \frac{1}{n+\frac{1}{j}s(u,j)}\log \frac{1}{M^h_n}+\frac{1-\frac{1}{j}}{n+\frac{1}{j}s(u,j)}
\log \frac{1}{\overline{m}_h}+ \frac{1}{jn+s(u, j)}\log \frac{1}{M} \\
&+  \frac{n(j-1)}{jn+s(u,j)}\cdot \frac{1}{n}
\log D+\frac{n}{n+ \frac{1}{j}s(u,j)}\cdot\frac{1}{n}\log A_{n}(z).\\
\end{split}
\end{equation*}
Without loss of generality, assume $\log A_{n}(z)> 0.$
By a simple calculation, we obtain  that 
\begin{equation}\label{r1}
  \frac{1}{n}\log A_{n}(z)<10\epsilon
\end{equation}
for all $n>N_1, z\in [u]$, for any $u\in B_n(X)$.

 Next we will show that there exists $N_2\in \N$ such that
 \begin{equation}\label{r2}
-6\epsilon < \frac{1}{n}\log A_{n}(z)
\end{equation}
for all $n>N_2, z\in [u]$ for any $u\in B_n(X)$. 
Define $f_{0}(x):=1$. 
Let $\overline{M}:=\max_{0\leq i\leq p}\{ f_{i}(x):x\in X\}$ and 
$\overline{m}_1:= \min_{0\leq k\leq p} \{e^{(S_k h) (x)}: x \in X\}.$
Take $N_2$ so that
\begin{equation}\label{use2}
\frac{1}{n}\vert \log (MM^h_n)\vert<\epsilon,
\frac{1}{n}\vert \log \left(\frac{\overline{M} e^{2C}}{\overline{m}_1}\right) \vert<\epsilon, \frac{n}{n+p}>\frac{1}{2}. 
\end{equation}
 for all $n>N_2$. For $n>N_2$,
let $u\in B_n(X)$. Construct $x\in [uw_1uw_2\dots uw_{j-1}u], j\geq 2,$ as in the above argument and let $z\in [u]$.
Using \ref{a0}, it is easy to obtain for each $j\geq 2$
\begin{equation}\label{use3}
f_{jn+\sum_{i=1}^{j-1}l_i}(x)\leq (\overline{M}e^{2C})^{j-1}(Mf_n(z))^j
\end{equation}
and 
\begin{equation}\label{use4}
e^{(S_{jn+\sum_{i=1}^{j-1}l_i}h)(x)}\geq \left(\frac{e^{(S_{n}h)(z)}}{M^h_n}\right)^{j}(\overline{m}_1)^{j-1}.
\end{equation}
Define $x^{*}\in X$ as before. For all $j\geq t(x^{*})$, by using (\ref{use2}), (\ref{use3}) and (\ref{use4}), 
we obtain
\begin{equation*}
-\epsilon< \frac{1}{jn+s(u, j)} \log A_{jn+s(u, j)} (x^{*})
< 2\epsilon +\frac{n}{n+\frac{1}{j}s(u,j)}\cdot \frac{1}{n}\log A_{n}(z).
\end{equation*}
Without loss of generality, assuming that  $\log A_{n}(z)<0$, we obtain (\ref{r2})
for all $n>N_2$, each $z\in [u], u\in B_n(X)$.
The result follows by (\ref{r1}) and (\ref{r2}).
\end{proof}

\begin{lemma} \label{simplel}
Under the assumptions of Theorem \ref{wk=g},  $\F$ is asymptotically additive if and only if 
 there exists a continuous function for which 
$m$ is an invariant weak Gibbs measure. 
\end{lemma}
\begin{proof}
Suppose $\F$ is asymptotically additive. Then by \cite[Theorem 1.2]{Cu} there exist $h, u_n\in C(X)$, $n\in\N$ such that $f_n(x)=e^{(S_nh)(x)+u_n(x)}$ satisfying $\lim_{n\rightarrow \infty}({1}/{n})\vert \vert u_n \vert\vert_{\infty}  =0$. 
Since there exist a constant $C>0$  such that
\begin{equation}\label{forg}
\frac{1}{C} \leq \frac{m [x_1\dots x_n]}{e^{-nP(\F)}f_n(x)} \leq {C}
\end{equation}
for each $x\in  [x_1\dots x_n]$,
replacing $f_n (x)$ by $e^{(S_nh)(x)+u_n(x)}$, we obtain
\begin{equation*}
\frac{1}{Ce^{\vert \vert u_n \vert \vert _{\infty}}}\leq \frac{e^{u_n(x)}}{C}
\leq \frac{m [x_1\dots x_n]}{e^{-nP(\F)+(S_nh)(x)}} \leq Ce^{u_n(x)}\leq Ce^{\vert \vert u_n \vert \vert _{\infty}}.
\end{equation*}
Set $A_n=Ce^{\vert \vert u_n \vert \vert _{\infty}}$. 
Since $ \lim_{n\rightarrow\infty}(1/n)\int \log f_n d\mu=\int hd\mu$
for every $\mu\in M(X, \sigma_X)$, we obtain that  $P(\F)=P(h)$. 
Conversely,
assume that $m$ is an invariant weak Gibbs measure 
for $\tilde h\in C(X)$. 
Hence there exists $C_n>0$ such
\begin{equation}\label{wk}
\frac{1}{C_n} \leq \frac{m [x_1\dots x_n]}{e^{-nP(\tilde{h})+(S_n\tilde{h})(x)}}\leq {C_n}
\end{equation}
for all $x \in [x_1\dots x_n]$, where $\lim_{n\rightarrow\infty}(1/n)\log C_n=0$.
Since $m$ is the Gibbs measure for $\F$, 
\begin{equation}\label{sg}
\frac{1}{C} \leq \frac{m [x_1\dots x_n]}{e^{-nP(\F)}f_n(x)} \leq C
\end{equation}
for some $C>0$. Using (\ref{wk}) and (\ref{sg}), we obtain 
\begin{equation} \label{property}
\frac{1}{C_nC}\leq  \frac{f_n(x)}{e^{(S_n({\tilde h-P(\tilde h)+P(\F))})(x)}}\leq C_nC,
\end{equation}
for all $x \in [x_1\dots x_n]$.
Hence by \cite[Theorem 1.2]{Cu} $\F$ is an asymptotically additive sequence. 
\end{proof}
\textbf{Proof of Theorem \ref{wk=g}.}
By \cite[Theorem 1.2]{Cu}, \ref{wka} implies \ref{wk0}. Theorem \ref{wk=g} follows by Lemma \ref{mainot} and Lemma \ref{simplel}.\\


\begin{thm} \label{cc}
Let $(X,\sigma_X)$ be a subshift and $\F=\{\log f_n\}_{n=1}^{\infty}$ be a sequence on $X$ satisfying \ref{a0} and \ref{a1} with bounded variation. 
Let $m$ be the unique invariant Gibbs measure for $\F$.  
Suppose that one of the equivalent statements in Theorem \ref{wk=g} holds. 
Then the following statements hold.
\begin{enumerate} [label=(\roman*)]
\item 
 There exits a sequence $\{C_{n, m}\}_{n,m\in\N}$ such that \label{estim}

\begin{equation} \label{fwg}
\frac{1}{C_{n,m}}\leq \frac{f_{n+m}(x)}{f_n(x) f_{m}(\sigma^n_X x)}\leq C_{n,m},
\textnormal{ where} \lim_{n\rightarrow \infty}\frac{1}{n}\log C_{n,m}= \lim_{m\rightarrow \infty}\frac{1}{m}\log C_{n,m}=0.
\end{equation}

\item If $m$ is a Gibbs measure for a continuous function, then $\F$ is an almost additive sequence on $X$. \label{qb}
\end{enumerate}
Hence if there is no sequence $\{C_{n, m}\}_{n,m\in\N}$ satisfying (\ref{fwg}), then there exists no continuous function for which $m$ is an invariant weak Gibbs measure. 
\end{thm}  
\begin{rem}
(1) In Example \ref{exshin}, we study a sequence which satisfies \ref{a0} and \ref{a1} without  (\ref{fwg}). (2) See \cite[Theorem 1.14 (ii)]{BCJP} for the result related to \ref{estim}.
\ref{qb} was also obtained in Section 4.1 \cite{Cu} since
 $m$ satisfies the quasi Bernoulli property (see  \cite{Cu}).
 \end{rem}
 \begin{proof}
 Let $h$ be defined as in Theorem \ref{wk=g} \ref{wk0}. By the proofs of Lemmas \ref{mainot} and  \ref{simplel}, we obtain that $P(\F)=P(h)$ and $m$ is an invariant weak Gibbs measure for $h$. Replacing $\tilde h$ by $h$ in  (\ref{property}), 
 we obtain that
\begin{equation} \label{weaks}
\frac{1}{C^3C_nC_mC_{n+m}}\leq \frac{f_{n+m}(x)}{f_n(x) f_{m}(\sigma^n_X x)}\leq C^3C_{n+m}C_nC_m.
\end {equation}
Since $\lim_{n\rightarrow \infty} (1/n)\log C_n=0$,  by setting $C_{n,m}:=C^3C_nC_mC_{n+m}$ we obtain the first statement. 
To obtain the second statement, we apply the latter part of the proof of Lemma \ref{simplel}. By replacing $C_n$ in (\ref{wk}) and (\ref{property}) by a constant, we obtain the second statement. The last statement follows from Theorem \ref{wk=g}.
\end {proof}


\section{relation between the existence of a continuous compensation function and an asymptotically additive sequence}\label{relation}



In this section, we consider relative pressure functions $P(\sigma_X, \pi, f)$, where $f\in C(X)$. In general  we can represent $P(\sigma_X, \pi, f)$   by using a subadditive sequence satisfying \ref{a4}. 
What are necessary and sufficient conditions for the existence of $h\in C(Y)$ satisfying $\int P(\sigma_X, \pi, f) dm=\int h dm$ for each $m\in M(Y, \sigma_Y)$? 
By \cite[Theorem 2.1]{Cu},  if  $P(\sigma_X, \pi, f)$ is represented  by an asymptotically additive sequence then we can find such a function $h$. We will study necessary conditions for the existence of such a function $h$ and relate them with the
existence of a compensation function for a factor map between subshifts. To this end, we will apply the results from Section \ref{subas}. We will study the property for periodic points from Lemma \ref{maina}\ref{ms2}.   

\begin{thm}\label{relativeff} \cite{LW} [Relativised Variational Principle]
Let $(X, \sigma_X)$ and $(Y, \sigma_Y)$ be subshifts and $\pi:X\rightarrow Y$ be a one-block factor map. Let $f\in C(X)$. Then for $m\in M(Y, \sigma_Y)$,
\begin{equation}\label{relativef}
\int P(\sigma_X, \pi,f) dm=
 \sup\{h_{\mu}(\sigma_X)-h_{m}(\sigma_Y) +\int f d\mu: \mu\in M(X, \sigma_X),\pi\mu=m\}.
 \end{equation}
\end{thm}

Applying the relativised variational principle, we first study Borel measurable compensation functions for factor maps between subshifts.
\begin{prop}\label{key1}
Let $(X, \sigma_X)$ and $(Y, \sigma_Y)$ be subshifts and $\pi:X\rightarrow Y$ be a one-block factor map. For each $f\in C(X)$,  
$f-P(\sigma_X, \pi, f)\circ \pi$ is a Borel measurable compensation function for $\pi$.
\end{prop}
\begin{rem} 
In general, $f-P(\sigma_X, \pi, f)\circ \pi$ is not  continuous on $X$.
\end{rem}
\begin{proof}
Let $m\in M(Y, \sigma_Y)$ and $\phi\in C(X)$. Applying Theorem \ref{relativeff}, we obtain
\begin{equation*} 
\begin{split} 
&\sup\{h_{\mu}(\sigma_X) -\int P(\sigma_X, \pi, f)\circ\pi  
d\mu +\int f d\mu + \int \phi\circ \pi d\mu: \mu\in M(X, \sigma_X), \pi \mu=m\}\\ 
&= \sup\{h_{\mu}(\sigma_X) +\int f d\mu: \mu\in M(X, \sigma_X), \pi \mu=m\} -\int P(\sigma_X, \pi, f)dm+  \int \phi dm\\
&=h_{m}(\sigma_Y)+\int \phi dm.
\end{split}
\end{equation*} 
Taking the supremum over $m \in M(Y, \sigma_Y)$, we obtain 
\begin{equation}\label{eqrel}
\begin{split} 
&\sup\{h_{\mu}(\sigma_X) -\int P(\sigma_X, \pi, f)\circ\pi  
d\mu +\int f d\mu + \int \phi\circ \pi d\mu: \mu\in M(X, \sigma_X)\}\\ 
&=\sup \{ h_{m}(\sigma_Y)+\int \phi dm: \mu\in M(Y, \sigma_Y)\}.
\end{split}
\end{equation} 
\end{proof}

Let $\pi:X\rightarrow Y$ be a one-block factor map between subshifts.
For $y=(y_i)_{i=1}^{\infty}$, let $E_n(y)$ be a set consisting of exactly one point from each cylinder $[x_1\dots x_n]$ in $X$ such that $\pi(x_1\dots x_n)=y_1\dots y_n$. For $n\in\N$ and $f\in C(X)$, define 
\begin{equation}\label{beq}
g_n(y)=\sup_{E_n(y)}\{\sum_{x\in E_n(y)}e^{(S_nf)(x)}\}.
\end{equation}


The following result can be deduced by \cite[Proposition 3.7(i)]{Fe1}.
If $(X, \sigma_X)$ and $(Y, \sigma_Y)$ be subshifts and $\pi:X\rightarrow Y$ is a one-block factor map, then for $f\in C(X)$,   
\begin{equation}\label{pesh}
P(\sigma_X, \pi, f)(y)=\limsup _{n\rightarrow\infty} \frac{1}{n} \log {g_n}(y)
\end{equation}
$\mu$-almost everywhere  for every invariant Borel probability measure $\mu$ on $Y$.
The equation (\ref{pesh}) was shown by Petersen and Shin \cite{PS} for the case when $X$ is an irreducible shift of finite type. The result for general subshifts 
is obtained by combining \cite[Proposition 3.7(i)]{Fe1} and the fact that 
$P(\sigma_X,  \pi, f)(y)\leq \limsup _{n\rightarrow\infty} (1/{n}) \log {g_n}(y)$ for all $y\in Y$. 
Note that the function $P(\sigma_X, \pi, f)$ is bounded on $Y$.

\begin{lemma}\label{new}
Let $(X, \sigma_X)$ be a subshift with the weak specification property, $(Y, \sigma_Y)$ be a subshift and $\pi:X\rightarrow Y$ be a one-block factor map. If $f\in C(X)$, then the sequence $\G=\{\log g_n\}_{n=1}^{\infty}$ on $Y$ satisfies \ref{a0} and \ref{a4} with bounded variation.  
\end{lemma}
\begin{rem}\label{pastover}
In particular, the sequence $\{\log g_n\}_{n=1}^{\infty}$ on $Y$ satisfies \ref{a0} and \ref{a1} with bounded variation if $f\in C(X)$ is in the Bowen class (see \cite{Fe1, Y3, ily}).   
\end{rem}
\begin{proof}
First we show that $\G=\{\log g_n\}_{n=1}^{\infty}$ satisfies \ref{a0}. Let $y=(y_1,\dots, y_n, \dots, $\\$ y_{n+m}, \dots)\in Y$. For each $x\in E_{n+m}(y)$, define $S_{x}$ by
$S_x:=\{x'\in E_{n+m}(y): x'_i=x_i , 1\leq i\leq n\}$. Take a point  $x^{*}\in S_x$ such that $e^{(S_nf)(x^{*})}=\max\{ e^{(S_nf)(z)}: z\in S_x\}$. Then we can construct a set $E_n(y)$ such that $x^{*}\in E_n(y)$. In a similar manner, for each $x\in E_{n+m}(y)$, define $S_{\sigma^{n}x}$ by $S_{\sigma^{n}x}:=\{x'\in E_{n+m}(y): x'_i=x_i , n+1\leq i\leq m+n\}$ and  take a point  $x^{**}\in S_{\sigma^nx}$ such that $e^{(S_mf)(\sigma^nx^{**})}=\max\{ e^{(S_mf)(z)}: z\in S_{\sigma^nx}\}$. Then we can construct a set $E_{m}(\sigma^n y)$ such that $\sigma^n x^{**}\in E_{m}(\sigma^n y)$. Hence  we obtain 
$g_{n+m}(y)\leq g_n(y)g_{m}(\sigma^ny)$. Next we show that $\G=\{\log g_n\}_{n=1}^{\infty}$ satisfies \ref{a4}. We modify slightly the arguments found in  \cite{ily} (see also \cite {Fe1}) by taking into account of tempered variation of $f$ and we write a proof for completeness.  Given  given $u\in B_n(Y)$ and  $v \in B_m(Y)$, let 
$x_1\dots x_n\in B_n(X)$ such that $\pi(x_1\dots x_n)=u$ and let 
$z_1\dots z_m\in B_m(X)$ such that  $\pi(z_1\dots z_m)=v$. Let $p$ be a weak specification number of $X$. Then there exists $\tilde {w}\in B_k(X)$, $0\leq k\leq p$ such that $x_1\dots x_n\tilde{w}z_1\dots z_m\in B_{n+m+k}(X)$. 
Hence if $x\in [x_1\dots x_n \tilde {w} z_1\dots z_m]$, by letting $\overline m=\min_{0\leq k\leq p}\{e^{(S_kf)(x)}:x\in X\}$, where $e^{(S_0f)(x)}:=1$ for all $x\in X$, we obtain
\begin{equation}\label{conti}
e^{(S_{n+k+m}f)(x)}\geq \overline {m} e^{(S_nf)(x)} e^{(S_mf)(\sigma^{n+k}x)}.
\end{equation}
 For $n\in \N$ let $M_n:=\sup\{e^{(S_nf)(x)}/e^{(S_nf)(x')}: x_i=x'_i, 1\leq i\leq n\}$. 
Since $X$ has the weak specification, $Y$ also satisfies the weak specification
property with a specification number $p$. 
Define $S$ by $S=\{w\in B_k(Y): 0\leq k\leq p, uwv\in B(Y)\}$ and let $y_{w}$ be a point from the cylinder set $[uwv]$.
Then 
\begin{align*}
&\sum_{w\in S} \sum_{x\in E_{n+m+\vert w\vert}(y_{w})}e^{(S_{n+m+\vert w \vert}f)(x)}\geq 
\sum_{\substack{x\in [x_1\dots x_n\tilde {w} z_1\dots z_m],\\ \pi(x_1\dots x_n \tilde {w} z_1\dots z_m)\in[uwv]}}\overline {m}e^{(S_nf)(x)} e^{(S_mf)(\sigma^{n+k}x)}\\
&\geq \frac{\overline m}{M_nM_m}(\sum_{\pi(x_1\dots x_n)=u}\sup_{x\in [x_1\dots x_n]}e^{(S_nf)(x)})
(\sum_{\pi(z_1\dots z_m)=v}\sup_{z\in [z_1\dots z_m]}e^{(S_mf)(z)}) \\
&\geq \frac{\overline m}{M_nM_m}\sup\{g_n(y): y\in[u]\}\sup\{g_m(y): y\in[v]\}.
\end{align*}
Hence 
\begin{equation*}
\sum_{w\in S}g_{n+m+\vert w\vert}(y_w)\geq \frac{\overline  {m}}{M_nM_m}\sup\{g_n(y): y\in[u]\}\sup\{g_m(y): y\in[v]\}.
\end{equation*}
Hence there exits $\bar{w}\in S$ such that 
\begin{equation*}
g_{n+m+\vert \bar w\vert}(y_{\bar w})\geq \frac{\overline {m}}{M_nM_m\vert S\vert}\sup\{g_n(y): y\in[u]\}\sup\{g_n(y): y\in[v]\}.
\end{equation*}
If $Y$ is a subshift on $l$ symbols, then $\vert S\vert \leq l^p$. Hence 
$\G$ satisfies \ref{a4} by setting $D_{n,m}=\overline {m}/(l^p M_nM_m)$. By the definition of $\G$, clearly $\G$ has bounded variation. 
\end{proof}

\begin{lemma} \cite{Wa} \label{sm}
Let $(X, \sigma_X)$ and $(Y, \sigma_Y)$ be subshifts and $\pi:X\rightarrow Y$ be a one-block factor map. Given $f\in C(X)$,  the following statements are equivalent for $h\in C(Y)$. 
\begin{enumerate}[label=(\roman*)]
\item $f-h\circ\pi$ is a compensation function for $\pi$.
\label{i2}
\item $\int P(\sigma_X, \pi, f-h\circ\pi)dm=0$ for each $m\in M(Y, \sigma_Y)$.\label{i3}
\item  $m(\{y\in Y: P(\sigma_X, \pi, f-h\circ \pi)(y)=0\})=1$ for each $m\in M(Y, \sigma_Y)$,
 \label{sm4}
\end{enumerate}
 \end{lemma}

\begin{lemma}\label{sm1}
Let $(X, \sigma_X)$ and $(Y, \sigma_Y)$ be subshifts and $\pi:X\rightarrow Y$ be a one-block factor map. Given $f\in C(X)$,  the following statement for $h\in C(Y)$ is equivalent to the equivalent statements in Lemma \ref{sm}.
\begin{enumerate}[label=(\roman*)]
\item 
$\int P(\sigma_X, \pi, f) dm=\int h dm$ for each $m\in M(Y, \sigma_Y)$.\\
\label{i1}
\end{enumerate}
 \end{lemma}
\begin{proof}
 Suppose that  Lemma \ref{sm1} \ref{i1} holds. Then (\ref{eqrel}) implies that $f-h\circ \pi$ is a compensation function for $\pi$. 
 Suppose that  Lemma \ref{sm} \ref{sm4} holds.
 Then (\ref{pesh}) implies that  for $m$-almost everywhere
 \[
P(\sigma_X, \pi, f-h\circ\pi)(y)=
\limsup_{n\rightarrow\infty} \frac{1}{n} \log \Big(\frac{g_n(y)}{e^{(S_nh)(y)}}\Big),
\]
where $g_n(y)$  is defined as in (\ref{beq}).
 Since $\{\log g_n\}_{n=1}^{\infty}$ is subadditive, 
$\{\log (g_n/e^{S_nh})\}_{n=1}^{\infty}$ is subadditive. 
 Applying 
 the subadditive ergodic theorem, 
 we obtain for each $m\in M(Y, \sigma_Y)$
\[ 
\begin{split} 
\int P(\sigma_X, \pi, f-h\circ \pi) dm 
 &=\int \limsup _{n\rightarrow\infty} \frac{1}{n} \log \Big (\frac{g_n}{e^{S_nh}} \Big )dm 
=\lim _{n\rightarrow\infty} \frac{1}{n}  \int  \log \Big (\frac{g_n}{e^{S_nh}}\Big )dm\\
&= \lim _{n\rightarrow\infty}  \frac{1}{n} \int \log {g_n}dm -\int h dm=0.
\end{split} 
\]
Hence we obtain Lemma \ref{sm1} \ref{i1}.
 \end {proof}
\begin{lemma} \label{use} 
Let $m\in M(Y, \sigma_Y)$. Then  
\[
P(\sigma_X, \pi, f-h\circ\pi)(y)=
\lim_{n\rightarrow\infty} \frac{1}{n} \log \Big (\frac{g_n (y)}{e^{(S_n h) (y)}} \Big)
\]
for $m$-almost everywhere on $Y$. 
 \end{lemma}
 \begin{proof}
 The result follows by the subadditive ergodic theorem.
\end{proof}

The main result of this section is the next theorem which relates the existence of a continuous compensation function for a factor map with the asymptotically additive property of the sequences $\G=\{\log g_n\}_{n=1}^{\infty}$. 
Given $f\in C(X)$, we continue to use $g_n$ as defined in equation (\ref{beq}). 

\begin{thm}\label{kohiyo}
Let $(X,\sigma_X)$ be an irreducible shift of finite type and $(Y,\sigma_Y)$
be a subshift. Let $\pi: X\rightarrow Y$ be a one-block factor map and  $f\in C(X)$. Then the following statements are equivalent for $h\in C(Y)$.

\begin{enumerate}[label=(\roman*)]

\item 
$P(\sigma_X, \pi, f-h\circ\pi)(y)=0$ 
for every periodic point $y\in Y$, equivalently, 
$$\lim_{n\rightarrow\infty} \frac{1}{n} \log \Big (\frac{g_n (y)}{e^{(S_n h) (y)}} \Big)=0$$ for every periodic point $y \in Y$.
\label{equi11}

\item 
The function $f-h\circ\pi$ is a compensation function for $\pi$. \label{equi1}

\item 
$$\lim_{n\rightarrow\infty} \frac{1}{n} \log \Big (\frac{g_n (y)}{e^{(S_n h) (y)}} \Big)=0$$
for every $y\in Y$. \label{equi2}

\item The sequence $\G=\{\log g_n\}_{n=1}^{\infty}$ is asymptotically additive on $Y$
satisfying 
$$\lim_{n\rightarrow \infty} \frac{1}{n} \lVert \log 
\Big (\frac{g_n}{e^{(S_n h)}}\Big) \rVert_{\infty}=0.$$
 \label{equi3}

\item 
$\int P(\sigma_X, \pi, f)dm=\int hdm$ for all $m\in M(Y, \sigma_Y)$.
 \label{wwd}

\end{enumerate}

\end{thm}
\begin{rem}\label{comments}
(1) 
 Theorem \ref{kohiyo} \ref{equi11} with $f=0$ is equivalent to the condition found by Shin \cite[Theorem 3.5]{Sh3} for the existence of a saturated compensation function between two-sided irreducible shifts of finite type (see Section \ref{apli}).  Hence by \cite[Theorem 3.5]{Sh3} Theorem \ref{kohiyo}  \ref{equi11}, \ref{equi1} and \ref{wwd} are equivalent when $f=0$ for a factor map between two-sided irreducible shifts of finite type.
By the result of Cuneo \cite[Theorem 2.1]{Cu}, if $\G$ is asymptotically additive then \ref{wwd} holds for some $h\in C(Y)$. 
(2) See Section  \ref{apli} for some examples and properties of $h$. 

\end{rem}
\begin{proof}
It is clear that \ref{equi3} implies \ref{equi2}.  By Lemma \ref{sm}, \ref{equi2} implies  \ref{equi1} and \ref{equi1} implies \ref{equi11}. Now we show that \ref{equi11} implies \ref{equi3}.
Suppose that \ref{equi11} holds. 
It is enough to show that Lemma \ref{maina} \ref{ms2} holds. 
Let $X$ be an irreducible shift of finite type on a set $S$ of finitely many symbols and $k$ be a weak specification number of $X$.  Let $L$ be the cardinality of the set $S$.
 Let $y=(y_1, y_2,\dots, y_{n},\dots )\in Y$. For a fixed $n\geq 3$, let $y_1=a$, $y_{n}=b$. Then $\pi^{-1}(y_1)=\{a_1, \dots, a_{L_1}\}$, where $a_i\in S$ for $1\leq i\leq L_1$, for some $L_1\leq L$, and $\pi^{-1}(y_{n})=\{b_1, \dots, b_{L_2}\}$ where $b_j\in S$ for $1\leq j\leq L_2$, for some $L_2\leq L$.
Define $W_{ij}:=\{a_{i}x_2\dots x_{n-1}b_{j}\in B_{n}(X): \pi(a_{i}x_2\dots x_{n-1}b_{j})=y_1\dots y_{n}\}$. Let $E^{i,j}_{n}(y)$ be a set 
consisting of exactly one point from each cylinder set $[u]$ of length $n$ of $X$, where $u\in W_{ij}$. Define $C_{i,j}:=\sum_{x\in E^{i,j}_{n}(y)} e^{(S_{n}f)(x)}$ and $M_{n}:=\sup\{e^{(S_nf)(x)}/e^{(S_nf)(y)}: x_i=y_i, 1\leq i\leq n\}$. If $W_{i,j}=\emptyset$, then define $C_{i,j}:=0$. Then 
\begin{equation*}
g_{n}(y)\geq \sum_{1\leq i\leq L_1, 1\leq j\leq L_2} C_{i,j}\geq \frac{1}{M_{n}}g_{n}(y),
\end{equation*}
where in the second equality, we use the fact that for any $E_n(y)$
\begin{equation*}   
\frac{g_n(y)}{M_n} \leq  \sum_{x\in E_{n}(y)} e^{(S_{n}f)(x)}.
\end{equation*}
Hence there exist $i_0, j_0$ such that 
\begin{equation}\label{kk}
 C_{i_0,j_0}\geq \frac{1}{L_1L_2 M_{n}}g_{n}(y)\geq \frac{1}{L^2 M_{n}}g_{n}(y).
\end{equation}
Note that $ (i_0, j_0)$ depends on $n$.
There exists an allowable word $w=w_1\dots w_{q}$ of length $q$ in $X$, $0\leq q \leq k$ such that $b_{j_0}wa_{i_0}$ is an allowable word of $X$.
Take an allowable word $a_{i_0}x_2\dots x_{n-1}b_{j_0}\in W_{i_0, j_0}$. Since $X$ is an irreducible shift of finite type,  we obtain a periodic point $\tilde{x}:=(a_{i_0},x_2,\dots, x_{n-1},b_{j_0}, w_1,\dots, w_{q})^{\infty}\in X$. Let $\pi(w_i)=d_i$ for each $i=1, \dots, q$. Let $y^{*}:=\pi(\tilde{x})$. Then
$y^{*}=(y_1,\dots, y_{n},d_1,\dots d_q)^{\infty}$ is a periodic point of $\sigma_Y$.

 
 For a fixed $n\geq 3$, define $P_0:=E^{i_0,j_0}_{n}(y)$. Define $P_1$ 
 by \begin{equation*}
 P_1=\{z=(z_i)_{i=1}^{\infty}\in X: z_{1}\dots z_{n} \in W_{i_0, j_0},  z_{n+1}\dots z_{n+q}=w, \sigma^{n+q} z=z \}.
\end{equation*}
 Observe that if $z\in P_1$, then $\pi(z)=y^{*}$ and $P_1$ is a set consisting of
 exactly one point from each cylinder $[u]$ of length $(n+q)$ of $X$ such that $\pi(u)= y_1\dots y_{n}d_1\dots d_q$  satisfying $u_1 \dots u_n \in W_{i_0, j_0}$ and  $u_{n+1}\dots u_{n+q}=w$.
 Then 
\begin{equation*}
 g_{n+q} (y^*)=\sup_{E_{n+q}(y^*)}\{\sum_{x\in E_{n+q}(y^*)}e^{(S_{n+q}f)(x)}\}
 \geq
 \sum_{x\in P_1}e^{(S_{n+q}f)(x)}
 \geq 
\frac{e^{m}}{M_n}( \sum_{x\in P_0}e^{(S_{n}f)(x)}),
\end{equation*}
where $m:= \min_{0\leq i\leq k}\{e^{(S_if)(x)}: x\in X\}$, $(S_0f)(x):=1$ for every $x\in X$. Next define  $P_2$ by
 \begin{equation*}
 \begin{split}
 P_2=&\{z=(z_i)_{i=1}^{\infty}\in X: \text{ for each } j=0,1, z_{j(n+q)+1}\dots z_{n(j+1)+jq} \in W_{i_0, j_0}.\\
 &  z_{(j+1)n+jq+1}\dots z_{(j+1)(n+q)}=w, \sigma^{2(n+q)} z=z \}.
 \end{split}
\end{equation*}
Observe that if $z\in P_2$, then $\pi(z)=y^{*}$ and $P_2$ is a set consisting of
 one point from each cylinder $[u]$ of length $(2n+2q)$ of $X$ such that $\pi(u)= y_1\dots y_{n}d_1\dots d_qy_1\dots y_{n}d_1\dots d_q$ satisfying $u_1\dots u_n, u_{n+q+1}\dots u_{2n+q}\in W_{i_0, j_0}$ and  $u_{n+1}\dots u_{n+q}=u_{2n+q+1}\dots u_{2n+2q}=w$.
Hence 
 \begin{equation*}
 g_{2(n+q)} (y^*)=\sup_{E_{2(n+q)}(y^*)}\{\sum_{x\in E_{2(n+q)}(y^*)}e^{(S_{2(n+q)}f)(x)}\}
 \geq
 \sum_{x\in P_2}e^{(S_{2(n+q)}f)(x)}
 \geq 
\frac{e^{2m}}{M^2_{n}}( \sum_{x\in P_0}e^{(S_{n}f)(x)})^2.
\end{equation*}
 Applying (\ref{kk}), 
 we obtain 
 \begin{equation*}
 g_{2(n+q)} (y^*)\geq \frac{e^{2m}}{M^2_{n}}(\sum_{x\in P_0}e^{(S_{n}f)(x)})^2
 \geq \frac{e^{2m}g^2_{n}(y^*)}{L^4M^4_{n}}.
 \end{equation*}
 Similarly, for $j\geq 3$,
define the set $P_j$ of periodic points by
 \begin{equation*}
 \begin{split}
 P_j&=\{z=(z_i)_{i=1}^{\infty}\in X: \text{for each } 0\leq l \leq j-1, z_{l(n+q)+1}\dots z_{(l+1)n+lq}\in W_{i_0, j_0}, 
 \\
 &z_{(l+1)n+lq+1}\dots z_{(l+1)(n+q)}=w,\sigma^{j(n+q)} z=z \}. 
 \end{split}
\end{equation*}
 If $z\in P_j$, then $\pi(z)=y^{*}$ and $P_j$ is a set consisting of
 one point from each cylinder $[u]$ of length $j(n+q)$ such that $\pi(u)= (y_1\dots y_{n}d_1\dots d_q)^{j}$  satisfying $u_{l(n+q)+1}=a_{i_0}$, $u_{(l+1)n+lq}=b_{j_0}$ 
 and $u_{(l+1)n+lq+1}\dots u_{(l+1)(n+q)}=w$ for each $0\leq l\leq j-1$. Then we obtain
 \begin{equation*}\label{ij}
  g_{j(n+q)}(y^*)=\sup_{E_{j(n+q)}(y^*)}\{\sum_{x\in E_{j(n+q)}(y^*)}e^{(S_{j(n+q)}f)(x)}\}
 \geq
 \sum_{x\in P_j}e^{(S_{j(n+q)}f)(x)}
 \geq  \frac{e^{jm}}{M^j_{n}} ( \sum_{x\in P_0}e^{(S_{n}f)(x)})^j.
\end{equation*}
 Applying (\ref{kk}), 
 we obtain 
 \begin{equation*}
 g_{j(n+q)}(y^*)\geq \frac{e^{jm}}{M^j_{n}} (\sum_{x\in P_0}e^{(S_{n}f)(x)})^j
 \geq (\frac{e^{m}}{L^2{M^2_{n}}})^jg^j_{n}(y^*).
\end{equation*}
 Since the function $g_n$ is locally constant, for $n\geq 3$, 
\begin{equation*}
 g_{j(n+q)}(y^*)\geq (\frac{e^{m}}{L^2{M^2_{n}}})^j\sup\{g_n(z): z\in [y_1\dots y_n]\}^j
\end{equation*}
for every $j\in \N$. 
 Hence the condition \ref{ms2} in Lemma \ref{maina} holds.
Applying Lemma \ref{maina}, we obtain \ref{equi3}. Finally, \ref{equi3} $\implies$ \ref{wwd} is immediate and \ref{wwd} implies \ref{equi1} by Lemma \ref{sm1}. 
\end{proof}


 

 Recall that if $f\in C(X)$ is in  the Bowen class, then $\G=\{\log g_n\}_{n=1}^{\infty}$ satisfies 
 \ref{a0} and \ref{a1} and  $\G$ has the unique Gibbs equilibrium state.

\begin{coro}\label{chara1} 
Under the assumptions of Theorem \ref{kohiyo}, assume also that  $f\in C(X)$ is a function in the Bowen class and let $m$ be the unique Gibbs equilibrium state for $\G=\{\log g_n\}_{n=1}^{\infty}$. Suppose that one of the equivalent statements in Theorem \ref{kohiyo} holds. 
Then

\begin{enumerate} [label=(\roman*)]
\item $m$ is an invariant weak Gibbs measure for $h$. \label{hweak}
\item The equation (\ref{fwg}) holds by replacing $f_n$ by $g_n$.
\item If $m$ is a Gibbs measure for a continuous function, then $\G$ is almost additive. 
\end{enumerate} 
Hence if there is no sequence $\{C_{n, m}\}_{n,m\in\N}$ satisfying (\ref{fwg}) by replacing $f_n$ by $g_n$, then there does not exist  a continuous function $h$ on $Y$ such that 
$$\int P(\sigma_X, \pi, f) d\mu=\int h d\mu$$ for every $\mu\in M(Y, \sigma_Y).$
 \end{coro}
 \begin{proof}
 Since 
$\lim_{n\rightarrow\infty}(1/n)\vert\vert \log (g_n/e^{S_nh})\vert\vert_{\infty}=0$,  applying the first part of the proof of Lemma \ref{simplel}, $m$ is an invariant weak Gibbs measure for $h$.  To show the second statement, we use similar arguments as in the proof of Theorem \ref{cc}.
To show the third statement, we apply the proof of Theorem \ref{cc}. The last statement is obvious by Theorem \ref{kohiyo}.
  \end {proof}
 \begin{rem} Applying Theorem \ref{maina}, we can study Theorem \ref{kohiyo} under a more general setting. Let $(X,\sigma_X), (Y,\sigma_Y)$ be subshifts and $\pi: X\rightarrow Y$ be a one-block factor map. Given a function $f\in C(X)$, suppose that $\G=\{\log g_n\}_{n=1}^{\infty}$ satisfies 
Theorem \ref{mainas} \ref{mss2}. Then Theorem \ref{kohiyo} holds. It would be interesting to study the conditions on factor maps $\pi$ satisfying Theorem \ref{mainas} \ref{mss2} for $\G$. 
\end{rem}



\section{Applications}\label{apli}
In this section, we give some examples and  applications. Applying the results from the previous sections, we study  the existence of a saturated compensation function for a factor map  between subshifts and factors of weak Gibbs measures for continuous functions.

\subsection{Existence of continuous saturated compensation functions}\label{shinsan}
Let $(X,\sigma_X)$ be an irreducible shift of finite type, $Y$ be a subshift and $\pi: X\rightarrow Y$ be a one-block factor map. For $n\in \N$, let  $\phi_n$ be the continuous function on $Y$ obtained by setting  $f=0$ in equation (\ref{beq}). Set $\Phi=\{\log \phi_n\}_{n=1}^{\infty}$.  

For a factor map $\pi$ between subshifts, there always exists a Borel measurable saturated compensation function $-P(\sigma_X, \pi, 0)\circ \pi$ given by a superadditive sequence $-\Phi\circ\pi$, however,  a continuous saturated compensation function does not always exist.  
Shin  \cite{Sh3} considered a one-block factor map $\pi: X\rightarrow Y$ between two-sided irreducible shifts of finite type and gave an equivalent condition for the existence of a saturated compensation function (see Theorem 3.5 in \cite{Sh3} for details). Note that  
the condition is equivalent to Theorem \ref{kohiyo} \ref{equi11} with $f=0$.

Here we characterize the existence of a saturated compensation function in terms of the type of the sequence $\Phi$ by applying Theorem \ref{kohiyo}.

\begin{coro}\label{A1}
Let $(X,\sigma_X)$ be an irreducible shift of finite type, $Y$ be a subshift and $\pi: X\rightarrow Y$ be a one-block factor map.
Then $-h\circ\pi, h\in C(Y)$ is a saturated compensation function if and only if one of the equivalent statements in Theorem \ref{kohiyo} holds with $f=0$. In particular, a saturated compensation function exists if and only if $\Phi$ is asymptotically additive on $Y$. 
If  $-h\circ\pi$ is a compensation function, then $h$ has the unique equilibrium state and it is a weak Gibbs measure for $h$. If there does not exist $\{C_{n,m}\}_{(n,m)\in \N \times \N}$ satisfying equation (\ref{fwg}) for $\Phi$, then there 
exists no continuous saturated compensation function for $\pi$.
\end{coro}
\begin{proof}
The result follows by setting $f=0$ in Theorem \ref{kohiyo} and Corollary \ref{chara1}.
\end{proof}
\begin{ex}\label{exshin} \cite{Sh3} \textbf {A sequence satisfying \ref{a0} and \ref{a1} which is not asymptotically additive.}
Shin \cite[Example 3.1]{Sh3} gave an example of a factor map $\pi: X\rightarrow Y$ between two sided irreducible shifts of finite type $X,Y$ without a saturated compensation function. We note that the  same results hold for one-sided subshifts. 
The sequence $\Phi=\{\log \phi_n\}_{n=1}^{\infty}$ is a subadditive sequence satisfying \ref{a0} and \ref{a1} with bounded variation and there exists a unique Gibbs equilibrium state $\nu$ for $\Phi$. 
Since there is no saturated compensation function, there does not exist a continuous function $h\in C(Y)$ such that 
$$\lim_{n\rightarrow\infty}\frac{1}{n}\int\log \phi_ndm=\int hdm$$
for every $m\in M(Y, \sigma_Y)$. Hence $\Phi$ is not an asymptotically additive sequence and there does not exist a continuous function on $Y$ for which $\nu$ is an invariant weak Gibbs measure (see Theorem \ref{weakgibbs}). 
Alternatively, a simple calculation shows that for any $x\in [12^m1]$ where $m\geq 3$ is odd,
$$\frac{\phi_{2+m}(x)}{\phi_{2}(x)\phi_{m}(\sigma^{2}x)}=\frac{\vert \pi^{-1}[12^m1]\vert}  {\vert\pi^{-1}[12]\vert \vert \pi^{-1}[2^{m-1}1]\vert}=\frac{1}{2^{\frac{m}{2}}+2}$$
(see \cite{Sh3}). 
Hence  for any sequence $\{C_{n,m}\}_{(n,m)\in \N \times \N}$ satisfying equation (\ref{fwg}) for $\Phi$, we obtain that $C_{2, m}\geq {2^{\frac{m}{2}}+2}$.
By Corollary \ref{A1}, there does not exist a continuous saturated compensation function.
\end{ex}
\begin{rem} \label{pflster}
 (1) Pfister and Sullivan \cite {PfS} studied a class of continuous functions satisfying bounded total oscillations on two-sided subshifts and showed that if a continuous function $f$ belongs to the class under a certain condition then an equilibrium state for $f$ is a weak Gibbs measure for some continuous function.
 Shin \cite[Proposition 3.5]{S1} gave an example of a saturated compensation function $G\circ \pi$ for a factor map $\pi: X\rightarrow Y$ between two-sided irreducible shifts of finite type $X,Y$ where $-G$ does not have bounded total oscillations. Let $(X^{+}, \sigma^{+}_{X})$ and  $(Y^{+}, \sigma^{+}_{Y})$ be the corresponding one-sided shifts of finite type and consider the factor map $\pi^{+}: X^{+}\rightarrow Y^{+}$.  Then 
 the corresponding saturated compensation function $G^{+}\circ \pi$ for $\pi^{+}, G^{+}\in C(Y^{+}),$  is obtained.
 Applying Theorem \ref{kohiyo} and Corollary \ref{chara1}, $-G^{+}$ has a unique equilibrium state  and it is a weak Gibbs measure for $-G^{+}$.
 (2) See Section 2 in \cite{BCJP} for examples of measures which are not weak Gibbs studied in quantum physics.

 \end{rem}

\begin{ex}\label{yukithesis} \textbf {A sequence satisfying \ref{a0} and \ref{a1} which is also asymptotically additive.} 
In \cite{Y1}, saturated compensation functions were studied to find the Hausdorff dimensions of some compact invariant sets of expanding maps of the torus. In \cite[Example 5.1]{Y1}, given a  factor map $\pi$ between topologically mixing subshifts of finite type $X$ and $Y$, a saturated compensation function $G\circ \pi$, $G\in C(Y)$, was found and $-G$ has a unique equilibrium state $\nu$ which is not Gibbs. Applying Theorem \ref{kohiyo} and Corollary \ref{chara1}, $\nu$ is an invariant weak Gibbs measure for $-G$. 
\end{ex}
\begin{rem}\label{rem2}
(1) In \cite{Fe1, Y2}, the ergodic measures of full Hausdorff dimension for some compact invariant sets of certain expanding maps of the torus were identified with equilibrium states for sequences of continuous function. If a saturated compensation function exists, then they are the equilibrium states of a constant multiple of a saturated compensation \cite{Y1}. 
(2) In Example \ref{yukithesis} \cite[Example 5.1]{Y1}, $X$ and $Y$ are one-sided subshifts of finite type.
Considering  the corresponding two-sided shifts of finite type $\hat X, \hat Y$ and the factor map $\hat \pi$ between them,  a saturated compensation function $\hat G\circ \pi$ for $\hat \pi$, $\hat G\in C(\hat Y)$, is obtained in the same manner as $G$ is obtained. 
The function $-\hat G$ on $\hat Y$ does not have bounded total oscillations (see Remark \ref{pflster} (1)). 
\end{rem}
\subsection{Factors of invariant weak Gibbs measures} \label{factg}
Factors of invariant Gibbs measures for continuous functions have been widely studied (see for example \cite{CU1,CU2, Fe1, K,Pi, Pi2, PK, V, Y2, Y3, Yo}). For a survey of the study of factors of Gibbs measures, see the paper by Boyle and Petersen \cite{BP}.
In this section, more generally, we study the properties of factors of invariant weak Gibbs measures. 
Given a one-block factor map $\pi: X\rightarrow Y$, 
and $f\in C(X)$,
define $g_n$ for each $n\in\N$ as in (\ref{beq}) and $\G=\{\log g_n\}_{n=1}^{\infty} $ on $Y$.


\begin{thm}\label{general}
Let $(X,\sigma_X)$ be an irreducible shift of finite type, $Y$ be a subshift and $\pi: X\rightarrow Y$ be a one-block factor map.
Suppose there exists $\mu$ such that $\mu$ is an invariant weak Gibbs measure for $f\in C(X)$. Then $\pi\mu$ is an invariant weak Gibbs measure for $\G=\{\log g_n\}_{n=1}^{\infty}$ on $Y$.
There exists $h\in C(Y)$ such that $\lim_{n\rightarrow \infty}(1/{n})\int\log g_n dm =\int h dm$ for all $m\in M(Y,\sigma_Y)$ if and only if one of the equivalent statements in Theorem \ref{kohiyo} \ref{equi11}-\ref{equi3} holds. 
Moreover, such a function $h$ exists if and only if  the invariant measure $\pi\mu$ is a weak Gibbs measure for a continuous function on $Y$. \end{thm}
\begin {rem}
(1) If $f$ is in the Bowen class, then there is a unique Gibbs equilibrium state for $\G$ and Corollary \ref{chara1} also applies. (2) If there exists $\mu$ such that $\mu$ is an invariant weak Gibbs measure for $f\in C(X)$, then $\pi\mu$ is an equilibrium state for $\G$.
\end {rem}
\begin{proof}
To prove the first statement, we apply the similar arguments as in the proof of \cite[Theorem 3.7]{Y3} and we outline the proof. Suppose that $f\in C(X)$ has an invariant weak Gibbs measure $\mu$. Then there exists $C _n>0$ such that 
\begin{equation*}
\frac{1}{C_n} \leq \frac{\mu [x_1\dots x_n]}{e^{-nP(f)+(S_nf)(x)}} \leq {C_n}
\end{equation*}
for each $x\in  [x_1\dots x_n]$, where $\lim_{n\rightarrow\infty}(1/n)\log C_n=0$.
Since $f$ has tempered variation, if we let 
$$M_n=\sup\{\frac{e^{(S_nf)(x)}}{e^{(S_nf)(y)}}:x,y  \in X, x_i=y_i \textnormal{ for } 1 \leq i \leq n\},$$
then $\lim_{n\rightarrow\infty}(1/n)\log M_n=0$.
By using the definition of the topological pressure, we obtain that 
$P(f)=P(\G)$. 
Since $$\pi\mu[y_1\dots y_n]=\sum_{\substack{x_1\dots x_n\in B_n(X)\\ \pi(x_1\dots x_n)=y_1\dots y_n}}\mu[x_1\dots x_n],$$ using the similar arguments as in the proof of  \cite[Theorem 3.7 ]{Y3}, we obtain 
\begin{equation*}
\frac{1}{C_nM_n} \leq \frac{\pi\mu [y_1\dots y_n]}{e^{-nP(\G)}g_n(y)} \leq {C_nM_n}.
\end{equation*}
Hence $\pi\mu$ is an invariant weak Gibbs measure for $\G$. 
The second statement holds by Theorem \ref{kohiyo}.
Now we show the last statement. Suppose such $h$ exists. Modifying slightly the proof of Corollary \ref{chara1} \ref{hweak} taking into account the fact that  $\pi\mu$ is a weak Gibbs measure for $\G$, we obtain that $\pi\mu$ is a weak Gibbs measure for $h$. To see the reverse implication, suppose $\pi\mu$ is weak Gibbs for some $\tilde {h}$. Then there exists $A_n>0$ such that 
\begin{equation}\label{laste}
\frac{1}{A_n} \leq \frac{\pi\mu [y_1\dots y_n]}{e^{-nP(\tilde h)+(S_n\tilde h)(y)}} \leq {A_n}
\end{equation}
for each $y\in [y_1\cdots y_n]$, where  $\lim_{n\rightarrow\infty}(1/n)\log A_n=0$.
If we let $K_n=C_nM_n$, then the similar arguments as in the latter part of the proof of Lemma \ref{simplel} show that 
$$\frac{1}{K_nA_n}\leq \frac{g_n(y)}{e^{(S_n(\tilde{h}-P(\tilde{h})+P(\G)))(y)}}\leq K_nA_n.$$ for each $y\in [y_1\cdots y_n]$. Hence
$\G$ is asymptotically additive. Set $h=\tilde h-P(\tilde{h})+P(\G).$
\end{proof}
The proof of Theorem \ref{general} gives us the following result. 
\begin{thm}\label{weakgibbs}
Under the assumptions of Theorem \ref{kohiyo}, suppose there exists $\mu$ such that $\mu$  is an invariant weak Gibbs measure for $f\in C(X)$.  Then there exists $h\in C(Y)$  satisfying the equivalent statements in Theorem \ref{kohiyo} 
if and only if there exists a continuous function on $Y$ for which $\pi\mu$ is an invariant weak Gibbs measure on $Y$.
\end{thm}
 
\begin{coro}\label{cc1}
Under the assumptions of Theorem \ref{general}, if there is no sequence $\{C_{n,m}\}_{n, m\in\N}$ satisfying 
the equation (\ref{fwg}) by replacing $f_n$ by $g_n$, then 
there does not exist a continuous function $h$  on $Y$ such that 
$\lim_{n\rightarrow \infty} (1/{n})\int \log g_n dm=\int h dm$ for every $m\in M(Y, \sigma_Y)$. 
Hence there exists no continuous function on $Y$ for which $\pi\mu$ is an invariant  weak Gibbs measure on $Y$.
 \end{coro} 
\begin{proof}
Suppose there exists $h\in C(Y)$ such that $\lim_{n\rightarrow \infty} (1/{n})\int \log g_n dm=\int h dm$ for every $m\in M(Y, \sigma_Y)$. By Theorem \ref{general}, $\G$ is asymptotically additive and $\pi\mu$ is an invariant weak Gibbs measure for $h$.  Hence 
there exists $A_n>0$ such that  (\ref{laste}) holds for $h$
for each $y\in [y_1\cdots y_n]$, where  $\lim_{n\rightarrow\infty}(1/n)\log A_n=0$. 
Let $K_n$ defined as in the proof of Theorem \ref{general}. Using $P(h)=P(\G)$ and additivity of $\{S_n h\}_{n=1}^{\infty}$, we obtain 
$$\frac{1}{K_{n+m}A_{n+m}K_nA_nK_mA_m}\leq \frac{g_{n+m}(y)}{g_n(y) g_{m}(\sigma^n_Y y)}\leq K_{n+m}A_{n+m}K_nA_nK_mA_m.$$
Define $C_{n,m}:= K_{n+m}A_{n+m}K_nA_nK_mA_m$ for each $n,m \in\N$. Then 
$\lim_{n\rightarrow \infty}(1/n)\log C_{n,m}=\lim_{m\rightarrow \infty}(1/m)\log C_{n,m}=0$.
Hence the result follows from Theorem \ref{general}.
\end{proof}


 





{\it Acknowledgements.}
{The author was partly supported by CONICYT PIA ACT172001 and by  
196108 GI/C at the Universidad del B\'{i}o-B\'{i}o.}

\end{document}